\documentclass[12pt]{amsart}
\usepackage{amsfonts}
\usepackage{latexsym,amsmath,amsthm,amssymb,amsfonts,mathrsfs}
\usepackage{graphicx}
\usepackage[usenames, dvipsnames]{color}
\usepackage{tikz}
\usepackage{float}
\usepackage{ulem}

\numberwithin{equation}{section}

\newtheorem{thm}{Theorem}[section]

\newtheorem{lem}[thm]{Lemma}
\newtheorem{prop}[thm]{Proposition}

\theoremstyle{remark}
\newtheorem{rem}{Remark}[section]

\title[Backward Uniqueness of extrinsic geometric flow]{Backward Uniqueness of Extrinsic Geometric Flow in general ambient manifolds}
\author{Dasong Li, John Man Shun Ma}
\date{\today}

\begin{document}

\begin{abstract}
    In this paper we prove two backward uniqueness theorems for extrinsic geometric flow of possibly non-compact hypersurfaces in general ambient complete Riemannian manifolds. These are applicable to a wide range of extrinsic geometric flow, including the mean curvature flow, inverse mean curvature flow, $\alpha$-Gauss curvature flow, $\sigma_k$ flow and so on. 
\end{abstract}
\maketitle

\section{Introduction}
Let $M$ be a smooth oriented $n$-manifold and $(\overline M, \bar g)$ a smooth oriented $(n+1)$-dimensional Riemannian manifold. A solution to an extrinsic geometric flow is a 1-parameter family $F: M\times[0,T]\rightarrow \overline M$ of $n$-dimensional immersed hypersurfaces in $\overline M$ which satisfies 
\begin{equation} \label{eqn extrinsic flow}
    \frac{\partial F}{\partial t}=-f\mathfrak n(x,t),\quad \text{ for all } x \in M^n, t \in [0,T].
\end{equation}
Here $\mathfrak n(p,t)$ denotes the unit normal vector of $M_t$ and $f$ is some curvature function 
\begin{equation} \label{eqn f = q (kappa)}
    f (x, t) = q(  \kappa (x, t)),
\end{equation}
where $\kappa (x, t) = (\kappa_1 (x, t), \cdots, \kappa_n (x, t))$ is the vector of principal curvatures of $F(\cdot, t)$ at $(x, t)$. Throughout this paper we assume that $q: \Omega \to \mathbb R$ is a smooth function which satisfies the following two properties: 
\begin{itemize}
    \item [(i)] Let $S_n$ be the permutation group of $\{1, \cdots, n\}$, which acts on $\mathbb R^n$ by permuting the coefficients. We assume that $\Omega$ is $S_n$-invariant and $q$ is symmetric on $\Omega$:
    \begin{equation*}
        q(\lambda_1, \cdots, \lambda_n)= q (\lambda_{\sigma(1)}, \cdots, \lambda_{\sigma(n)} ), \ \ \ \text{ for all } \sigma \in S_n, (\lambda_1, \cdots, \lambda_n)\in \Omega.
    \end{equation*}
    \item [(ii)] $q$ is monotone: that is 
    \begin{equation} \label{eqn dfn q is monotone}
    \frac{\partial q}{\partial \lambda_i} >0, \ \ i=1, \cdots, n
    \end{equation}
    on $\Omega$. Note that the monotonicity assumption ensures that the flow (\ref{eqn extrinsic flow}) is parabolic. 
\end{itemize} 

Examples of most popular extrinsic geometric flows include 
\begin{itemize}
    \item the mean curvature flow (MCF): $f = \kappa_1 + \cdots +\kappa_n$, 
    \item the inverse mean curvature flow (IMCF): $f = -(\kappa_1 + \cdots +\kappa_n)^{-1}$, 
    \item the $\alpha$-Gauss curvature flow ($\alpha$-GCF): $f = \operatorname{sgn}(\alpha)(\kappa_1\cdots \kappa_n)^\alpha$.
\end{itemize}
and so on. 

For the past 40 years, extrinsic geometric flow had been one of the crucial tools in the study of submanifolds. To name some of the applications, MCF was used by Ilmanen-White \cite{IlmanenWhite} to derive sharp lower bounds on density of area-minimizing cones, and by Buzano-Haslhofer-Hershkovits \cite{BHH} to prove the connectedness of the moduli space of two-convex embedded spheres. Huisken-Ilmanen constructed in \cite{HuiskenIlmanen} a weak notion of IMCF, and this is used to prove the Riemannian Penrose Inequality. Other applications of IMCF include the computation of Yamabe invariants by Bray-Neves \cite{BN}, and the new proof of Ricci pinching conjecture by Huisken-Koerber \cite{HuiskenKoerber}. In general, various extrinsic geometric flow are employed to prove a lot of geometric inequalities \cite{BM}, \cite{KwongYong},



When $M$ is compact, under the symmetry and monontonicity assumptions on $q$, (\ref{eqn extrinsic flow}) is equivalent up to diffeomorphism to a uniform parabolic scalar equation. Hence given any initial immersion $F_0 : M \to \overline M$, there exists a unique short time solution $F :M \times [0,\epsilon] \to \overline M$ to (\ref{eqn extrinsic flow}) with $F(\cdot, 0) = F_0(\cdot)$. This is well-documented in the literature, see e.g. Theorem 3.1 in Huisken-Polden's survey \cite{HP}, or Chapter 18 in Andrew-Chow-Guenther-Langford's recent monograph \cite{ACGL}. 

When $M$ is non-compact, both the existence and uniqueness of solutions to (\ref{eqn extrinsic flow}) are not direct consequences of the standard PDE theory. In the following we discuss some of the known existence and uniqueness results in the context of non-compact MCF and IMCF.

In MCF, Ecker-Huisken constructed in \cite{EH_entire_graph} long time solutions to the MCF of entire graph in $\mathbb R^{n+1}$. The estimates derived therein are later localized and used in \cite{EH_interior_estimates} to prove the existence of solution to the MCF starting from a non-compact hypersurface in $\mathbb R^{n+1}$ which satisfies a uniform local Lipschitz condition. Daskalopoulos-Saez \cite{DS} proves the uniqueness of entire graphical solutions to the MCF under an one-sided bounds on the second fundamental form. In higher codimensions, long time existence and uniqueness to the MCF of entire graphs are also established in Chau-Chen-He \cite{CCH} in the Lagrangian case and Koch-Lamm \cite{KochLamm} with the assumption of small Lipschitz norm. Without the (global) graphical assumption, a strong uniqueness theorem was established by Chen-Yin \cite{ChenYin} in any codimension when the initial immersion has bounded second fundamental form $A$ and satisfies a uniform graphic condition. Lee and the second-named author \cite{LeeMa} use an energy method to prove a uniqueness theorem for MCF with possibly unbounded curvature in any codimensions. Very recently, Peachey \cite{Peachey} constructed a smooth complete metric $\bar g$ on $\mathbb R^2$ which admits two uniformly proper solutions to the curve shortening flow coming out from the $x$-axis. 

In IMCF, the non-compact case was studied by Daskalopoulos-Huisken \cite{DH}, where they prove the existence and uniqueness to the IMCF of entire convex graph and uniform finite time singular convergence for convex entire graphs with conical behavior at infinity. More recently, Choi-Daskalopoulos \cite{CD} constructed convex solutions to the IMCF starting from some $C^{1, 1}$ convex hypersurface.

In this work, we are interested in the problem of backward uniqueness: Given two solutions $F, \widetilde F$ to the extrinsic geometric flow (\ref{eqn extrinsic flow}) such that $F(\cdot, T) = \widetilde F(\cdot, T)$. Under what conditions do we have $F = \widetilde F$?

Some versions of backward uniqueness results were proved for the MCF and the IMCF. Lee and the second-named author proved in \cite{LeeMa} a backward uniqueness theorem for MCF in general ambient manifold in any codimension. This generalizes the previous result in \cite{HongHuang}, \cite{Zhang}, when the ambient manifold is Euclidean. For IMCF, Ho-Pyo proved in \cite{HoPyo} a backward uniqueness theorem for IMCF of compact mean convex hypersurfaces in $\mathbb R^{n+1}$. 

Our first result is the following backward uniqueness theorem for extrinsic geometric flow (\ref{eqn extrinsic flow}) of hypersurfaces in general ambient manifold for general speed $f$. 

\begin{thm}\label{thm backward uniqueness of extrinsic geometric flow}
Let $(\overline M, \bar g)$ be a smooth oriented $(n+1)$ dimensional complete Riemannian manifold with positive injectivity radius $i_0$ and uniform upper bounds $|\overline\nabla ^j \overline R| \le B_j$ for $j=0,1,2,3$, where $\overline R$ is the Riemann curvature tensor of $\bar g$. Let $F, \widetilde F : M \times [0,T] \to \overline M$ be two solutions to the extrinsic geometric flow (\ref{eqn extrinsic flow}), where $f$ satisfies (\ref{eqn f = q (kappa)}) for some symmetric, smooth and monotone curvature function $q : \Omega \to \mathbb R$. Assume the following holds: 
\begin{itemize}
    \item there is $L >0$ such that 
    \begin{equation} \label{eqn |A| le L}
        |A|_g + |\widetilde A|_{\tilde g} \le L  
    \end{equation}
    and 
    \begin{equation} \label{eqn |nabla A| le L}
    |\nabla A|_g + |\widetilde \nabla \widetilde A|_{\tilde g} \le L
    \end{equation}
    on $M \times [0,T]$, and
    \item there is $\delta >0$ such that 
    \begin{equation} \label{eqn dist (kappa, partial Omega) > delta}
        \operatorname{dist} (\kappa (x, t) , \partial \Omega) >\delta, \ \  \operatorname{dist} (\tilde\kappa (x, t) , \partial \Omega) > \delta
    \end{equation}
       for all $(x, t) \in M \times [0,T]$, where $\kappa (x, t)$, $\tilde \kappa (x, t)$ are respectively the vectors of principal curvatures of $F, \widetilde F$ at $(x, t)$.
\end{itemize}
If $F(\cdot, T) = \widetilde F(\cdot, T)$ is a complete immersion. Then $F(\cdot,t)=\widetilde F(\cdot,t)$ for all $t \in [0,T]$.
\end{thm}

We illustrate condition (\ref{eqn dist (kappa, partial Omega) > delta}) in the following examples: 
\begin{itemize}
    \item For MCF, the curvature function $q= \lambda_1 + \cdots + \lambda_n$ is monotone on $\Omega=\mathbb R^n$. Hence $\partial \Omega$ is empty and (\ref{eqn dist (kappa, partial Omega) > delta}) always hold.
    \item For IMCF, $q = -(\lambda_1 + \cdots + \lambda_n)^{-1}$ and $q$ is monotone on $\Omega = \{ \lambda \in \mathbb R^n: \lambda_1 + \cdots + \lambda _n >0\}$. Hence (\ref{eqn dist (kappa, partial Omega) > delta}) implies that $F, \widetilde F$ are both uniformly mean convex: $H, \widetilde H >\delta$. 
    \item For GCF, $q = \lambda_1\cdots \lambda_n$ is monotone on $\Omega =\{ \lambda : \lambda_i >0, \ i=1, \cdots, n\}$, and (\ref{eqn dist (kappa, partial Omega) > delta}) implies that $F, \widetilde F$ are both uniformly convex: $\kappa_i, \tilde\kappa_i >\delta$ for $i=1, \cdots, n$. 
\end{itemize}

When the flow is the MCF, Theorem \ref{thm backward uniqueness of extrinsic geometric flow} is weaker then the one in \cite{LeeMa}, where they prove the backward uniqueness for MCF in any codimension. In the IMCF case, our result generalizes Ho-Pyo's backward uniqueness result \cite{HoPyo} to non-compact $M$ and general ambient Riemannian manifold $(\overline M, \bar g)$. For other extrinsic geometric flows, this result is new to the author's knowledge. 

The condition (\ref{eqn |nabla A| le L}) is not necessary in the MCF case (see \cite[Theorem 1.4]{LeeMa}). However, for general extrinsic geometric flow (\ref{eqn extrinsic flow}), there's a term in the evolution equation for the second fundamental form $A$ which is quadratic in $\nabla A$, and this prevents us from applying the stardard Shi-type estimates as in the case for MCF. With a concavity assumption on the curvature function, and the convexity of the immersions, we can drop condition (\ref{eqn |nabla A| le L}). This is the content of our second main theorem: 

\begin{thm} \label{thm backward uniqueness of convex extrinsic geometric flow with concave speed}
Let $(\overline M, \bar g)$ be a smooth oriented $(n+1)$ dimensional complete Riemannian manifold with positive injectivity radius $i_0$ and uniform upper bounds $|\overline\nabla ^j \overline R| \le B_j$ for $j=0,1,2,3$, where $\overline R$ is the Riemann curvature tensor of $\bar g$. Let $F, \widetilde F : M \times [0,T] \to \overline M$ be two convex solutions to the extrinsic geometric flow (\ref{eqn extrinsic flow}), where $f$ satisfies (\ref{eqn f = q (kappa)}) for some symmetric, smooth and monotone and concave curvature function $q : \Omega \to \mathbb R$. Assume that $F$, $\widetilde F$ satisfy (\ref{eqn |A| le L}), (\ref{eqn dist (kappa, partial Omega) > delta}) for all $(x, t) \in M \times [0,T]$. If $F(\cdot, T) = \widetilde F(\cdot, T)$ is a complete immersion. Then $F(\cdot,t)=\widetilde F(\cdot,t)$ for all $t \in [0,T]$.
\end{thm}

The above Theorem is applicable to e.g. the GCF, and the IMCF of convex hypersurfaces.  


Next we discuss the proof of Theorem \ref{thm backward uniqueness of extrinsic geometric flow} and Theorem \ref{thm backward uniqueness of convex extrinsic geometric flow with concave speed}. In \cite{KotschwarBU}, Kotschwar proves a general backward uniqueness theorem for time dependence sections of a vector bundle over a non-compact manifold which satisfies an ODE-PDE inequalities. This is used in \cite{KotschwarBU} to derive backward uniqueness theorem for various intrinsic geometric flows including the Ricci flow and the $L^2$-curvature flow. One can apply Kotschwar's theorem to obtain backward uniqueness results for MCF (as in \cite{HongHuang}, \cite{Zhang}) and IMCF (in \cite{HoPyo}) in Euclidean spaces. However, for a general ambient background $(\overline M, \bar g)$, geometric quantities like the Gauss map of an immersion are defined as section of the pullback bundles. Hence for two apriori different immersions $F, \widetilde F$, one cannot directly compare those quantities. This difficulty was addressed in \cite{LeeMa} (and in \cite{McGahagan}, \cite{SongWang} in different contexts) by constructing a parallel transport $P$ from $\widetilde F$ to $F$, and the authors are able to prove various uniqueness and backward uniqueness theorem for MCF in general ambient manifold. We will apply the same method here. We remark that although the second fundamental form (and its derivatives) are tensors defined on the domain $M$ in the hypersurface case, we still need to estimate the full gauss map, and the use of the parallel transport $P$ cannot be avoided. Another complications in our proof, when compared to \cite{LeeMa}, \cite{HoPyo}, is that our higher derivative estimates of the speed $f$ is not immediate. Although the condition (\ref{eqn dist (kappa, partial Omega) > delta}) is stated in terms of the principal curvatures, we perform the calculations in term of the second fundamental form $A= (h_{ij})$ and the metric $g = (g_{ij})$. By choosing for each $x\in M$ a suitable local coordinates, we are able to bound not only $(\tilde h_{ij}, \tilde g_{ij})$ of $\widetilde F$, $(h_{ij}, g_{ij})$ of $F$, but also the difference of the higher order derivatives of $q (\tilde \kappa) - q(\kappa)$ in terms of $\tilde h_{ij}- h_{ij}$ and $\tilde g_{ij}-g_{ij}$, which is crucial in our calculations.  

The organization of the paper is as follows. In section 2, we introduce the background material. This includes some standard knowledge on the curvature function, submanifold theory and the evolution equations associated to the extrinsic geometric flow (\ref{eqn extrinsic flow}). In section 3, we prove two Shi-type estimates (Theorem \ref{thm higher order gradient estimates of A} and Theorem \ref{thm curvature estimates when F is convex and q concave}) for the higher derivatives of the second fundamental form. In section 4, we describe the technique used in \cite{LeeMa} to compare geometric quantities of two a-priori different immersions via the parallel transport. In section 5, we derive the key time derivative estimates of different geometric quantities associated to the flow. In section 6, we prove the crucial ODE-PDE inequalities and finish the proof of Theorem \ref{thm backward uniqueness of extrinsic geometric flow} and Theorem \ref{thm backward uniqueness of convex extrinsic geometric flow with concave speed}. In the last section, we state and prove a backward uniqueness theorem (Theorem \ref{thm backward uniqueness compact case}) for compact immersions which holds for any ambient Riemannian manifold $(\overline M, \bar g)$.


\section{Background}
In this section, we review some background materials on the curvature function $q$ and the evolution equations associated to (\ref{eqn extrinsic flow}). 

\subsection{The curvature function}
We start by defining the curvature function and proving some of its basic properties. The main reference is \cite{AMZ} (see also \cite[Chapter 18]{ACGL}). Let $S(n)$ (resp. $S_+(n)$) be the space of $n\times n$ symmetric (resp. positive definite) matrices. Let $\Omega_{SS}$ be an open subset of $S(n) \times S_+(n)$ that is invariant in the sense that  
$$ (X, g) \in \Omega_{SS} \Rightarrow (\Lambda X\Lambda^T , \Lambda g\Lambda^T)\in \Omega_{SS}, \ \ \forall \Lambda \in GL(n, \mathbb R). $$
Let $Q : \Omega_{SS} \to \mathbb R$ be a smooth function such that 
\begin{equation} \label{eqn invariance of Q} 
Q(X, g)=Q(\Lambda X\Lambda^T , \Lambda g\Lambda^T), \ \ \forall (X, g)\in \Omega_{SS},\ \ \Lambda \in GL(n, \mathbb R).
\end{equation}
Setting $\Lambda = g^{-1/2}$, we have 
$$ Q(X, g) = Q(g^{-1/2} X g^{-1/2} , I).$$
Set  $\Omega_S = \Omega_{SS} \cap S (n) \times \{ I\} \cong S(n)$, then $\Omega_S$ is open in $S(n)$. Let $Q_S : \Omega _S\to \mathbb R$ be defined by 
$$ Q_S(X) = Q(X, I).$$
Then $Q_S$ is smooth and $Q_S(OAO^T)=Q_S(A)$ for any $O \in O(n)$. For each $A\in S(n)$, let $\lambda (A)$ be the multi-set of its eigenvalues. To be precise, $\lambda(A)$ is in $\mathbb R^n /S_n$. Let $\Omega$ be a subset of $\mathbb R^n$ defined by 
$$ \Omega = \left\{ \lambda = (\lambda_1, \cdots, \lambda_n) : D_{\lambda} = \begin{bmatrix} \lambda _1 & & \\ & \ddots & \\ && \lambda_n \end{bmatrix} \in \Omega_S \right\} .$$
By a theorem of Schwarz \cite{Schwarz}, there is a symmetric smooth function $q : \Omega \to \mathbb R$ such that 
$$Q_S(A) = q (\lambda (A)).$$
On the other hand, given a smooth symmetric function $q: \Omega \to \mathbb R$ on a $S_n$-invariant open susbset $\Omega$ of $\mathbb R^n$, one defines $Q, Q_S$ as 
\begin{equation*}
    Q_S(X) := q (\lambda(X)), \ \ Q(h, g) := Q_S( \mathcal D(h, g)) = Q_S (g^{-1/2} h g^{-1/2}).
\end{equation*}
Hence from now on we call both $q$ and $Q$ the curvature function of the flow (\ref{eqn extrinsic flow}).

The following can be checked using the compactness of $O(n)$. 
\begin{lem}
    Let $K$ be a compact subset of $\Omega$. Then 
    \begin{equation} \label{eqn dfn K_S}
        K_S:= \{ A \in \Omega_S : \lambda(A) \in K\}
    \end{equation}
    is a compact subset of $\Omega_S$. 
\end{lem}


Let $\mathcal D: \Omega _{SS}\to \Omega_S$ be given by $\mathcal D(X, g) = g^{-1/2}Xg^{-1/2}$. While $\mathcal D^{-1} (K_S)$ is not compact, we have the following 

\begin{prop} \label{prop K_{SS} is a compact set}
 Let $K_S $ be a compact subset in $\Omega_S$ and $\lambda >1$. Then 
 \begin{equation} \label{eqn dfn of K_lambda}
     K_{SS}^\lambda :=\mathcal D^{-1} (K_S) \cap \{ (X, g) \in \Omega_{SS}: \lambda ^{-1}\delta_0 \le g\le \lambda \delta_0\}
 \end{equation}
 is compact in $\Omega_{SS}$. Here $\delta_0$ is the Euclidean metric on $\mathbb R^n$. 
\end{prop}

\begin{proof}
Let $(X_k, g_k)_{k=1}^\infty$ be a sequence in $K_{SS}^\lambda$. Since $K_S$ is compact, a subsequence of $(g^{-1/2}_k X_k g^{-1/2}_k)_{k=1}^\infty$ converges. Since $\lambda^{-1} \delta_0 \le  g_k\le \lambda \delta_0$, a subsequence of $(g_k)_{k=1}^\infty$ converges. Using 
$$ X_k = g^{1/2}_k (g^{-1/2}_k X_k g^{-1/2}_k) g^{1/2}_k,$$
one concludes that a subsequence of $(X_k)_{k=1}^\infty$ also converges. Hence $(X_k, g_k)_{k=1}^\infty$ has a convergent subsequence and $K_{SS}^\lambda$ is compact. 
\end{proof}

Define the matrix $L_Q = (L^{ij}_Q)$, where
\begin{equation} \label{eqn dfn of L^ij_Q}
    L^{ij}_Q :=\frac{\partial Q}{\partial h_{ij}}
\end{equation} 
and $\frac{\partial Q}{\partial h_{ij}}$ is the partial derivative of $Q = Q(h_{ij}, g_{ij})$ with respect to $h_{ij}$. It is proved (see \cite[Chapter 18]{ACGL}) that the monotonicity assumption (\ref{eqn dfn q is monotone}) on $q$ is equivalent to that $L_Q$ is positive definite on $\Omega_{SS}$: 
\begin{equation} \label{eqn L_Q is positive definite}
    L_Q >0 , \ \ \text{ for all } (h, g) \in \Omega_{SS}. 
\end{equation}

Lastly, recall that $q$ is concave on $\Omega$ if and only if $Q$ satisfies 
\begin{equation} \label{eqn q concave iff Q concave}
    \frac{\partial^2 Q}{\partial h_{ij} \partial h_{kl}} \le 0, \ \ \text{ on } \Omega_{SS}.
\end{equation}
A proof can be found in \cite[Corollary 18.9]{ACGL}.

\subsection{Extrinsic geometric flow: the evolution equations}
Let $F: M\times [0,T]\to \overline M$ be a family of immersed hypersurfaces in $\overline M$. Let $F_t = \frac{\partial F}{\partial t}$ and for any local coordinates $(x^1, \cdots, x^n)$ of $M$, we write $F_i = \frac{\partial F}{\partial x^i}$. When no confusion would arise, for each $t\in [0,T]$ we also use $F$ to denote the immersion $F (\cdot, t)$ at $t$.  

Let $\mathfrak n = \mathfrak n(x, t)$ be a unit normal vector fields along $F$. We assume that $\mathfrak n$ is chosen such that for any oriented local charts $(x^1, \cdots, x^n)$ of $M$, 
$$ (F_1, \cdots, F_n , \mathfrak n)$$
is positively oriented with respect to the orientation of $\overline M$. 

The second fundamental form $A$ of $F$ is defined by 
\begin{equation*}
    A(X, Y) = -\bar g ( \overline \nabla _X Y, \mathfrak n),
\end{equation*}
and in local coordinates we write $A = (h_{ij})$. 




Next we consider the evolution equations of several geometric quantities along the flow (\ref{eqn extrinsic flow}). A proof can be found in \cite{HP}.

\begin{lem} \label{lem basic evolution of extrinsic flow}
    Let $F: M\times [0,T]\to \overline M$ satisfies the extrinsic flow (\ref{eqn extrinsic flow}). Let $g=g(t)$ denote the induced metric on $M$ and $\nabla$ the Levi-Civita connection with respect to $g$, and let ($h_{ij}$) be the second fundamental form of $F$. Then the following holds:
\begin{itemize}
\item [(1)] $\partial_t g_{ij}= 2f h_{ij}$, 
\item [(2)] $\partial_t \Gamma^{k}_{ij}=g^{kl}[\nabla_i(fh_{jl})+\nabla_j(fh_{il})-\nabla_l(fh_{ij})]$,
\item [(3)] $D_t \mathfrak n = F_*  \nabla f$,
\item [(4)] $\nabla_i \mathfrak n = g^{jk} h_{ij} F_k$.
\end{itemize}
\end{lem}


The evolution equation for the second fundamental form $A$ and the speed $f$ are also computed in \cite[Corollary 3.3]{HP}:

\begin{lem}\label{lem evolution of A under extrinsic geometric flow}
Let $F: M\times [0,T]\to \overline M$ satisfies the extrinsic flow (\ref{eqn extrinsic flow}), where $f$ satisfies (\ref{eqn f = q (kappa)}) for a smooth symmetric function $q$. Then the following holds:
\begin{align} \label{eqn evolution of A under extrinsic flow}
    \frac{\partial}{\partial t}h_{ij} &= \Lambda^{kl}\nabla_k\nabla_l h_{ij} + \frac{\partial ^2 Q}{\partial h_{kl} \partial h_{pq}} \nabla_ih_{kl} \nabla_j h_{pq} \\
    & \ \ \ + \Lambda^{kl} \{ -h_{kl} h_{i}^m h_{mj} +h_{k}^mh_{il}h_{mj} - h_{kj} h_{i}^m h_{ml} +h_{k}^m h_{ij} h_{ml}  \notag \\
    &\ \ \ + \bar R_{kilm}h^m_j + \bar R_{kijm}h^m_l + \bar R_{iljm}h^m_k + \bar R_{oioj}h_{kl} - \bar R_{okol}h_{ij} + \bar R_{mljk}h^m_i \notag \\ 
    &\ \ \ + \bar \nabla_k \bar R_{ojil} + \bar \nabla_i \bar R_{oljk} \}  - f (h_{ik} h_j^k +\bar R_{oioj}) \notag, \\
    \frac{\partial f}{\partial t} &= \Lambda^{ij} \nabla_i \nabla _j f -f \Lambda^{ij} (h_{i}^k h_{kj} + \bar R_{oioj}). \label{eqn evolution of f under extrinsic flow}
\end{align}
where $\overline R$ is the Riemann curvature tensor of $(\overline M, \bar g)$, $Q$ is the curvature function and $\Lambda ^{ij} := L_Q^{ij} (h, g)$. 
\end{lem}

In the above Lemma, for any ambient $(0,q)$ tensor $S = (S_{\alpha_1\cdots \alpha_{q}})$ on $\overline M$, $(S_{i_1\cdots i_q})$ is a $(0,q)$ tensor on $M$ defined by
\begin{equation*}
    S_{i_1\cdots i_q} := S (F_{i_1}, \cdots, F_{i_q}).
\end{equation*}
Also, an $o$ in the indices refers to evaluation with $\mathfrak n$, e.g.
\begin{equation*}
S_{oioj}:= S( \mathfrak n, F_{i}, \mathfrak n , F_{j}). 
\end{equation*}

In this paper we need to compare tensors that are defined along two a-priori different immersions $F$, $\widetilde F$ and we prefer notations that clearly show the dependence on $F$, $\widetilde F$. Hence we use the (also standard) pullback notation: 
\begin{equation} \label{eqn dfn of pullback of ambient tensor}
    (F^*S)_{i_1\cdots i_q} = S(F_{i_1}, \cdots, F_{i_q}).
\end{equation}
Although tensors of the form $S_{oioj}$ are not pullback of ambient tensors on $\overline M$, by abusing notations we still use $F^*( \iota^k_{\mathfrak n} S)$ to denote such tensors. For example, given a $(0,4)$ tensor $S$ on $\overline M$, $F^*(\iota^2_{\mathfrak n} S)$ denotes a $(0,2)$ tensor on $M$ of the form   
\begin{equation}\label{eqn dfn of pullback of ambient tensor with mathfrak n}
    (F^*(\iota^2_{\mathfrak n} S))_{ij} = S(\mathfrak n, F_i, \mathfrak n, F_j)
\end{equation}
or $S(F_i , \mathfrak n, \mathfrak n, F_j)$, $S(F_i, \mathfrak n, F_j, \mathfrak n)$ etc. 

\begin{rem}
    In the latter part of the paper we also use the restriction $S|_F$ of an ambient tensor $S$ (see \cite[Lemma 4.2]{LeeMa}), which is related but different from $F^*S$ defined in (\ref{eqn dfn of pullback of ambient tensor}).  
\end{rem}

In the last part of this section, we derive the evolution equations for the higher derivative of $A$ along the flow (\ref{eqn extrinsic flow}). To do this we first fix some notations. For any tensor $\mathcal T$ on $M$ and $k \in \mathbb N$, let $\mathcal T^k = \mathcal T\otimes \cdots \otimes \mathcal T$ be the $k$ times tensor product of $\mathcal T$. Given $j$ tensors $\mathcal T_1, \cdots, \mathcal T_j$ on $M$, we use
\begin{equation*}
    \mathcal L = \mathcal L(\mathcal T_1, \cdots, \mathcal T_j)
\end{equation*} 
to denote some finite linear combination of tensors of the form 
$$ \mathcal T_{j_1}^{k_1} * \cdots * \mathcal T_{j_m}^{k_m}, $$
where $\{ j_1, \cdots, j_m \} \subset \{1, \cdots, j\}$, $k_1, \cdots, k_m \in \mathbb N$ and $*$ are some contractions. 

Also, from the invariance (\ref{eqn invariance of Q}) of $Q$, for any symmetric two tensors $B$ and any Riemannian metric $g$ on $M$, 
\begin{equation} \label{eqn dfn of partial^k_h Q (B, g)}
    \partial^k_h Q (B,g)^{i_1 j_1 \cdots i_kj_k} := \frac{\partial^kQ}{\partial h_{i_1j_1} \cdots \partial h_{i_kj_k}} (B_{ij},g_{ij})
\end{equation}
defines a $(2k,0)$ tensor on $M$. Note that $\Lambda = \partial_h Q(A, g)$. 


Using the above notations, we can now rewrite (\ref{eqn evolution of A under extrinsic flow}) as 
\begin{equation} \label{eqn evolution eqn of A - simplified version}
    \left( \partial_t - \Lambda^{kl} \nabla_k \nabla_l \right) A =  \partial^2_h Q(A,g) * \nabla A * \nabla A+ \mathcal L, 
\end{equation}
where 
\begin{equation} \label{eqn dfn of mathcal L_0}
    \mathcal L= \mathcal L (Q(A, g), \partial_h Q (A,g), g^{-1}, A, F^*\overline R, F^*(\iota^2_{\mathfrak n} \overline R), F^*(\bar\nabla \overline R)). 
\end{equation}

In the following we define $\Delta^f = \Lambda^{ij} \nabla_i \nabla_j$. 

\begin{prop} \label{prop evolution equation of nabla^k A}
    For each $k=0,1,\cdots$, $\nabla^k A$ satisfies the equation 
    \begin{equation} \label{eqn evolution equation for nabla^k A}
        (\partial _t -\Delta^f) \nabla^k A = \partial^2 _h Q (A, g) * \nabla A * \nabla^{k+1} A + \mathcal L_k,
    \end{equation}
    where $\mathcal L_k$ is a finite linear combination of tensors of the form
    \begin{equation} \label{eqn finite linear combination in mathcal L_k}
        \mathscr F * \nabla^{i_1}  A * \cdots * \nabla^{i_j } A,
    \end{equation}
    where 
    \begin{equation*} 
    \mathscr F = \mathscr F( (\partial^l_h Q(A, g))_{l\le k+2}, g^{-1}, F^* (\iota_{\mathfrak n}^m \bar\nabla^j \overline R)_{j\le k+1, m\le j+2}),
    \end{equation*}
    $i_1, \cdots, i_j \in \{0, \cdots, k\}$ and $i_1+ \cdots +i_j \le k+2$. 
\end{prop}

\begin{proof}
    For any $k =0,1,2,\cdots $, note  
    \begin{align} \label{eqn [partial_t - Delta, nabla] = }
        (\partial _t - \Delta^f ) \nabla ^{k+1} A &= (\partial _t - \Delta^f ) \nabla \nabla^k A \\
        &= \nabla ( \partial_t-\Delta^f)\nabla^k A + [\partial_t , \nabla] \nabla^k A + [\nabla, \Delta^f] \nabla ^kA. \notag 
    \end{align}
    We need to calculate the commutators $[\partial_t , \nabla]$ and $[\nabla, \Delta^f]$. First recall that for any tensor $\mathcal T$ on $M$, 
\begin{equation} \label{eqn commutator of nabla and partial_t}
     [\partial_t, \nabla] \mathcal T := (\partial_t \nabla - \nabla \partial_t)\mathcal T = (\partial_t \Gamma)* \mathcal T = \mathscr F^1 * \mathcal T, 
\end{equation}
where 
$$\mathscr F^1 = g^{-1} * (Q * \nabla A + \partial_h Q *A*\nabla A)$$ 
by (2) in Lemma \ref{lem basic evolution of extrinsic flow}. 

Before calculating $[\nabla, \Delta^f] $ we remark that for any $(0,q)$ tensor $S$ on $\overline M$,
\begin{align} \label{eqn nabla of ambient tensors}
    \nabla F^*S &= F^*\overline\nabla S + F^*(\iota_nS) *A, \\
    \nabla (F^*(\iota^k_{\mathfrak n} S)) &= F^*(\iota^k_{\mathfrak n} \overline \nabla S) + F^*(\iota^{k-1}_{\mathfrak n} S) *A + F^*(\iota^{k+1}_{\mathfrak n} S) *A \notag
\end{align}
 by the definition of $A$ and (4) in Lemma \ref{lem basic evolution of extrinsic flow}. Next, applying $\nabla$ on both sides of the Gauss equation   
\begin{equation} \label{eqn Gauss equation}
R = F^*\overline R  + A*A
\end{equation}
and use (\ref{eqn nabla of ambient tensors}), we obtain
\begin{equation} \label{eqn nabla of Gauss equation}
    \nabla R = F^* \bar\nabla \overline R  + F^* (\iota_{\mathfrak n} \overline R)* A + \nabla A*A.  
\end{equation}
Now we calculate $[\nabla , \Delta^f]$. Note that 
\begin{equation} \label{eqn commuting Delta_f and nabla first step}
    \Delta ^f \nabla_k - \nabla_k \Delta^f = \Lambda^{ij} (\nabla_i \nabla_j \nabla_k - \nabla_k \nabla_i \nabla_j) - (\nabla_k \Lambda^{ij}) \nabla_i \nabla_j.  
\end{equation}
By definition of $\Lambda$ and since $\nabla g = 0$, $\nabla \Lambda = \partial^2_h Q (A,g) *\nabla A$. By the Ricci identity, for any tensor $\mathcal T$ on $M$,  
\begin{equation} \label{eqn commutator of Delta and nabla}
    (\nabla_i \nabla_j \nabla_k - \nabla_k \nabla_i \nabla_j) \mathcal T = \nabla R * \mathcal T + R * \nabla \mathcal T,
\end{equation}
where $R$ is the Riemann curvature tensor of $(M, g(t))$. Together with (\ref{eqn Gauss equation}), (\ref{eqn nabla of Gauss equation}) and (\ref{eqn commuting Delta_f and nabla first step}), one obtains
\begin{align} \label{eqn commutator of Delta and nabla, version 2}
    [\nabla, \Delta^f] \mathcal T &= \partial^2_h Q(A,g)* \nabla A * \nabla^2 \mathcal T + \partial _h Q * (F^* \overline R + A*A)* \nabla \mathcal T \\
    &\ \ \ + \partial_h Q (A,g) * ( F^* \bar\nabla \overline R  + F^* (\iota_{\mathfrak n} \overline R)* A + \nabla A *A) * \mathcal T  \notag \\
    &=: \partial^2_h Q(A,g)* \nabla A * \nabla^2 \mathcal T  + \mathscr F^2 * \nabla \mathcal T + \mathscr F^3 * \mathcal T. \notag
\end{align}
Hence by (\ref{eqn [partial_t - Delta, nabla] = }) and setting $\mathscr F^4 = \mathscr F^1 + \mathscr F^3$,
\begin{align} \label{eqn induction step for (partial_t -Delta) nabla ^k A}
        (\partial _t - \Delta^f ) \nabla ^{k+1} A &= \nabla ( \partial_t-\Delta^f)\nabla^k A + \partial^2_h Q (A, g) *\nabla A*\nabla^{k+2}A \\
        &\ \ \ + \mathscr F^2*\nabla^{k+1}A + \mathscr F^4*\nabla^kA \notag 
\end{align}
Now we are ready to prove by induction. The $k=0$ case follows from (\ref{eqn evolution eqn of A - simplified version}) and (\ref{eqn dfn of mathcal L_0}). Next assume that (\ref{eqn evolution equation for nabla^k A}) holds for some $k=0, 1, \cdots$. By (\ref{eqn induction step for (partial_t -Delta) nabla ^k A}) and the induction hypothesis, 
\begin{align*}
        (\partial _t - \Delta^f ) \nabla ^{k+1} A &= \nabla (\partial^2 _h Q (A, g) * \nabla A * \nabla ^{k+1} A+ \mathcal L_k) + \partial^2_h Q (A, g) *\nabla A*\nabla^{k+2}A \\
        &\ \ \ + \mathscr F^2*\nabla^{k+1}A + \mathscr F^4*\nabla^k A \notag \\
        &= \partial^2_h Q (A, g) *\nabla A*\nabla^{k+2}A + \nabla \mathcal L_k \notag \\
        & \ \ \ + \partial^3_h Q (A, g) * (\nabla A)^2 * \nabla ^{k+1}A + \partial^2 _h Q (A, g) * \nabla^2 A * \nabla^{k+1}A  \notag \\
        &\ \ \ + \mathscr F^2*\nabla^{k+1}A + \mathscr F^4*\nabla^k A \notag \\
        &=: \partial^2_h Q (A, g) *\nabla A*\nabla^{k+2}A + \mathcal L_{k+1} \notag.
\end{align*}
Using the induction hypothesis, $\nabla (\partial^l _h Q(A, g)) = \partial^{l+1}_h Q(A, g)* \nabla A$ and (\ref{eqn nabla of ambient tensors}), one sees that $\mathcal L_{k+1}$ is a finite linear combination of tensors of the form 
$$ \mathcal F * \nabla^{i_1} A * \cdots * \nabla^{i_j } A, $$
where $i_1, \cdots, i_j \in \{0,1,\cdots, k+1\}$ and $i_1 + \cdots +i_j \le k+3$. Hence the Proposition is proved by induction. 
\end{proof}

\section{Curvature estimates of extrinsic flow} Let $F: M\times [0,T] \to \overline M$ be a solution to the extrinsic geometric flow (\ref{eqn extrinsic flow}) such that 
\begin{equation} \label{eqn section 3 |A|le L}
    |A(x, t)| \le L,
\end{equation}
and 
\begin{equation} \label{eqn section 3 dist (kappa , partial Omega) ge delta}
    \operatorname{dist} (\kappa(x, t) , \partial \Omega) \ge \delta 
\end{equation}
on $M\times [0,T]$. 

Let $K \subset \Omega$ be the compact subset 
\begin{equation} \label{eqn dfn of K}
    K  = \{ \kappa \in \Omega : |\kappa| \le L \text{ and } \operatorname{dist} (\kappa, \partial \Omega\} \ge \delta\}.
\end{equation}
By (\ref{eqn section 3 |A|le L}), (\ref{eqn section 3 dist (kappa , partial Omega) ge delta}) we have $\kappa (x, t)\in K$ for all $(x, t) \in M\times [0,T]$. 

\begin{prop} \label{prop (h, g) in K_SS}
    Let $F: M \times [0,T]\to \overline M$ be a family of immersions which satisfies (\ref{eqn extrinsic flow}), (\ref{eqn section 3 |A|le L}) and (\ref{eqn section 3 dist (kappa , partial Omega) ge delta}). Then there is $\lambda>1$ depending only on $L, n, q, T, \delta$ and an oriented smooth atlas $\mathscr A$ of $M$ such that the following holds: for all $x\in M$, there is an oriented smooth local charts $(U, (x^1, \cdots, x^n))$ of $\mathscr A$ such that in this chart,
    \begin{equation*} 
    (h_{ij}(x, t), g_{ij}(x, t)) \in K^\lambda_{SS}
    \end{equation*}
    for all $t$, where $K_{SS}^\lambda$ is defined in (\ref{eqn dfn of K_lambda}). 
\end{prop}

\begin{proof}
Let $\mathscr A$ be a smooth atlas of $M$ such that for each smooth local chart $(U, (x^i))$ in $\mathscr A$, we have
$$ \frac{1}{2} \delta_0 \le g(x, T) \le 2\delta_0,$$
where $\delta_0$ is the Euclidean metric. Since $K$ is compact, $f(x, t) = q (\kappa (x, t))$ is uniformly bounded. By (1) in Lemma \ref{lem basic evolution of extrinsic flow} and that $|A|\le L$, $|\partial_t g_{ij}|$ is also uniformly bounded. Hence there is $\lambda_1 >0$ depending on $n$, $\delta$, $L$, $q$ and $T$ such that 
\begin{equation} \label{eqn uniform equivalence of metrics}
    \lambda_1^{-1} g (x, T) \le g (x, t) \le \lambda_1 g (x, T)
\end{equation}
for all $t \in [0,T]$ and all $x$ in the local chart. As a result, in $(U, (x^i))$ we have 
\begin{equation} \label{eqn g equivalent to delta_0} 
\lambda^{-1} \delta_0 \le g(x, t) \le  \lambda \delta_0,
\end{equation}
where $\lambda = 2\lambda_1$. Together with $\kappa (x, t)\in K$, we have $(h_{ij} (x, t) , g_{ij}(x, t)) \in K^\lambda_{SS}$.
\end{proof}

The goal of this section is to prove two higher derivative estimates of the second fundamental form. Let $B_0 (x_0, a)$ be a geodesic ball in $(M, g_0)$ centered at $x_0$ with radius $a>0$. By (\ref{eqn section 3 |A|le L}), the speed $|f|$ is uniformly bounded by $\overline C$. Hence the image of $F (B_0 (x_0, a) \times [0,T])$ lies in the geodesic ball $B=B_{\bar g} (p_0, a + \overline C T)$ of $(\overline M, \bar g)$, where $p_0  =F(x_0, 0)$.  

Let $B^{loc}_j$, $j=0, 1, \cdots, $ be the local curvature bounds of $(\overline M, \bar g)$ in $B$. That is,
\begin{equation} \label{eqn |bar nabla^k bar R| le B_k}
    |\overline \nabla ^j \overline R (p)| \le B^{loc}_j, \ \ \text{ on } B_{\bar g} (p_0 , a + \overline CT). 
\end{equation}

In the following, we assume that the calculations are done in oriented local charts described in Proposition \ref{prop (h, g) in K_SS}. Using that $K^\lambda_{SS}$ is compact, $q$ is monotone and (\ref{eqn g equivalent to delta_0}), there is $\theta >0$ such that 
    \begin{equation} \label{eqn Lambda is equivalent to g}
        \theta g^{-1}(x, t) \le \Lambda(x, t) \le \frac{1}{\theta} g^{-1} (x, t)
    \end{equation}
    for all $(x, t)$. Also, by compactness of $K^\lambda_{SS}$ there are positive constants $Q_l$, $l=0, 1, \cdots$, such that 
    \begin{equation} \label{eqn partial^l_h Q uniformly bounded by Q_k}
        |\partial^l_h Q(A, g)| \le Q_l,
    \end{equation}
    and the constants $Q_1, Q_2,\cdots$ depend only on $Q, L$ and $\delta$. In this section, we use $\widetilde C_k$, $k=0,1,\cdots$, to denote positive constants that depend only on $n, T, L, \delta, \theta$, $B_i$ for $i=0,1, \cdots, k+1$ and $Q_j$ for $j=0,1,\cdots, k+2$. The constants $\widetilde C_k$ might change from line to line.

\begin{thm} \label{thm higher order gradient estimates of A}
    Let $F: M \times [0,T] \to \overline M$ be a solution to the extrinsic geometric flow (\ref{eqn extrinsic flow}), where $f$ satisfies (\ref{eqn f = q (kappa)}) for some symmetric, smooth and monotone curvature function $q : \Omega \to \mathbb R$. Assume that $F$ satisfies (\ref{eqn section 3 |A|le L}), (\ref{eqn section 3 dist (kappa , partial Omega) ge delta}) and that $|\nabla A|\le L$. Then for each $k=2, 3, \cdots$, there is a positive constant $L_k$ depending on $k, n, L, \delta, T, Q, a$ and $B_j^{loc}$ in (\ref{eqn |bar nabla^k bar R| le B_k}) for $j=0, 1, \cdots, k+1$ such that 
    \begin{equation} \label{eqn higher order gradient estimates of h}
        |\nabla ^{k} A | \le \frac{L_k}{t^{\frac{k-1}{2}}}, \ \ \text{ on }B_0\left(x_0, \frac{a}{2}\right) \times (0,T].
    \end{equation} 
\end{thm}

\begin{proof}
    This is proved using Proposition \ref{prop evolution equation of nabla^k A} and is standard. We only derive the necessary inequalities, and refer the readers to the proof of \cite[Theorem 3.2]{ChenYin} for more details.  Using (\ref{eqn Lambda is equivalent to g}), (\ref{eqn partial^l_h Q uniformly bounded by Q_k}), $|A|\le L$, $|\nabla A|\le L$ and Proposition \ref{prop evolution equation of nabla^k A}, 
    \begin{align} \label{eqn evolution ineq for |G|^2}
        ( \partial_t-\Delta^f)|\nabla A|^2&=\partial_t g ^{-1}*(g^{-1})^2*(\nabla A)^2 + 2\langle(\partial_t-\Delta^f)\nabla A,\nabla A\rangle   \notag \\
        &\ \ \ -2\Lambda^{pq} \langle\nabla_p\nabla A, \nabla_q\nabla A\rangle  \notag\\ 
        &\le \widetilde C_1 ( |\nabla A|^2 +|\nabla ^2 A| +1) - 2\theta |\nabla^2 A|^2 \\ 
        &\le -\theta |\nabla^2 A|^2 + \widetilde C_1, \notag
    \end{align}
    Similarly we have
    \begin{align*}
        (\partial_t-\Delta^f)|\nabla ^2 A|^2&=\partial_ tg^{-1}*(g^{-1})^3* (\nabla^2 A)^2 + 2\langle(\partial_t-\Delta^f)\nabla^2 A,\nabla^2 A\rangle  \\
        &\ \ \ -2\Lambda^{pq} \langle \nabla_p \nabla^2 A, \nabla_q \nabla^2 A\rangle  \\ &\le \widetilde C_2 (|\nabla^2 A|^2 +|\nabla^2 A||\nabla^3 A| +|\nabla^2 A||\mathcal L_1|)- 2\theta |\nabla^3 A|^2\\
        &\le \widetilde C_2 (|\nabla^2 A|^2  +|\nabla^2 A||\mathcal L_2|)-\theta |\nabla^3 A|^2.
    \end{align*}
    By (\ref{eqn finite linear combination in mathcal L_k}) and $|A|\le L$, $|\nabla A|\le L$,
    $$|\mathcal L_2|\le \widetilde C_2 (|\nabla ^2 A|^2 + |\nabla^2 A| + 1).$$
    Hence
    \begin{align} \label{eqn evolution ineq for |nabla G|^2}
        (\partial_t-\Delta^f)|\nabla^2 A|^2 \le -\theta |\nabla^3 A|^2 + \widetilde C_2 (|\nabla^2 A|^2 +1). 
    \end{align}
    Based on the evolution inequality (\ref{eqn evolution ineq for |G|^2}), (\ref{eqn evolution ineq for |nabla G|^2}) and a cutoff argument from \cite{ChenZhu}, one can proved as in \cite{ChenYin} that
\begin{align*}
    |\nabla^2 A|\le \frac{L_2}{t^{1/2}}, \ \ \text{ on } B_0 \left(x_0, \left( \frac 12 + \frac 1{2^3}a\right)\right) \times (0,T],
\end{align*}
for some positive constant $L_1$. Next assume that for $i=2, \cdots, k$, 
\begin{align*}
    |\nabla^{i}A|\le \frac{L_i}{t^{\frac{i-1}{2}}} \ \ \ \text{ on } B_0 \left( x_0, \left( \frac 12 + \frac{1}{2^{i+1}} \right)a \right) \times (0,T],
\end{align*}
for some positive constants $L_i, i=2,\cdots,k$.

By Proposition \ref{prop evolution equation of nabla^k A} and the induction hypothesis we have
\begin{align}\label{eqn evolution inequality of |nabla^k G|^2}
    &\quad (\partial_t-\Delta^f)|\nabla^{k}A|^2 \notag \\
    =& \ \partial_tg^{-1}* (g^{-1})^{k+1} *(\nabla^{k} A)^2 \notag \\
    &+ 2 \langle (\partial_t-\Delta^f) \nabla^{k}A,\nabla^{k} A \rangle -2 \Lambda^{pq} \langle \nabla_p\nabla^{k}A ,\nabla_q\nabla^{k}A \rangle \notag  \\
    \le&\ \widetilde C_{k}|\nabla^{k} A|^2+ 2\langle \partial^2_hQ(A,g)*\nabla A*\nabla^{k+1} A+\mathcal L_{k}, \nabla^{k} A\rangle - 2\theta|\nabla^{k+1}A|^2 \\
    \le&\ \widetilde C_{k}|\nabla^{k} A|^2+ 2\langle \mathcal L_{k}, \nabla^{k} A\rangle - \theta|\nabla^{k+1} A|^2 \notag \\
    \le &\ \widetilde C_{k}|\nabla^{k}A|^2 - \theta |\nabla^{k+1} A|^2\notag \\ 
    &+ \widetilde C_{k} \sum_{ i_1,\cdots,i_j\le k, i_1+\cdots+i_j\le k+2}|\nabla^{i_1} A| \cdots |\nabla^{i_j} A||\nabla^{k}A|  \notag \\
    \le & - \theta|\nabla^{k+1}A|^2 + \frac{\widetilde C_{k}}{t^{k-1}} + \widetilde C_{k}\sum_{ i_1,\cdots,i_j\le k, i_1+\cdots+i_j\le k+2} \frac{1}{t^{(i_1 + \cdots + i_j-j+k)/2}}. \notag \\
    \le & - \theta|\nabla^{k+1}A|^2 + \frac{\widetilde C_{k}}{t^{k}}. \notag
\end{align}
Similarly, 
\begin{align}\label{eqn evolution inequality of |nabla^k+1 G|^2 1}
    ( \partial_t-\Delta^f)|\nabla^{k+1}A|^2 \le \widetilde C_{k+1} |\nabla^{k+1} A|^2+ 2\langle \mathcal L_{k+1}, \nabla^{k+1} A\rangle  - \theta |\nabla^{k+2}A|^2. 
    \end{align}
observing that 
\begin{align*}
    |\mathcal L_{k+2}|&\le \widetilde C_{k+1} \sum_{i_1,\cdots,i_j\le k+1, i_1+\cdots+i_j\le k+3}|\nabla^{i_1}A| \cdots |\nabla^{i_j} A|\notag \\ 
    &\le \widetilde C_{k+1} (|\nabla^2 A| + |\nabla A|^2 + |\nabla A|+1) |\nabla^{k+1} A| \\
    &\ \ \ + \widetilde C_{k+2}\sum_{i_1,\cdots,i_j\le k, i_1+\cdots+i_j\le k+3}|\nabla^{i_1} A| \cdots |\nabla^{i_j} A|  \notag\\
    &\le \widetilde C_{k+1} \left( \frac{|\nabla^{k+1} A|}{t^{1/2}} + \frac{1}{t^{\frac{k+2}{2}}}\right) \notag,
\end{align*}
together with (\ref{eqn evolution inequality of |nabla^k+1 G|^2 1}) we have 
\begin{align}\label{eqn evolution inequality of |nabal^k+1C|^2 2}
    (\frac{\partial}{\partial t}-\Delta^f)|\nabla^{k+1} A|^2\le -\theta |\nabla^{k+2} A|^2
    + \widetilde C_{k+1} \left( \frac{|\nabla^{k+1} A|^2}{t^{1/2}} + \frac{|\nabla^{k+1}A|}{t^{\frac{k+2}{2}}}\right).
\end{align}
Let 
\begin{align*}
    \psi=(\mathcal A+t^{k-1}|\nabla^{k}A|^2)t^k|\nabla^{k+1}A|^2,
\end{align*}
where $\mathcal A$ is a positive constant to be chosen later. By (\ref{eqn evolution inequality of |nabla^k G|^2}) and (\ref{eqn evolution inequality of |nabal^k+1C|^2 2}) we have
\begin{align*}
    (\partial_t-\Delta^f)\psi=& \frac{(k-1)t^{k-1}}{t}|\nabla^{k}A|^2(t^{k} |\nabla^{k+1} A|^2)\notag \\
    &+(\mathcal A+t^{k-1}|\nabla^{k} A|^2) k\frac{t^{k}}{t} |\nabla^{k+1}A|^2\notag \\
    &+t^{2k-1} |\nabla^{k+1} A|^2 (\partial_t-\Delta^f)|\nabla^{k}A|^2\notag \\
    &+(\mathcal A+t^{k-1}|\nabla^{k} A|^2) t^{k}(\partial_t-\Delta^f)|\nabla^{k+1} A|^2 \notag \\
    &-2t^{2k-1} \langle \nabla |\nabla^{k} A|^2, \nabla |\nabla^{k+1}A|^2\rangle \notag \\
    \le & \frac{2k-1}{t}\psi  + t^{2k-1}|\nabla^{k+1} A|^2(- \theta|\nabla^{k+1} A|^2 + \frac{\widetilde C_{k}}{t^{k}}) \\
    &+t^{k} (\mathcal A+t^{k-1}|\nabla^{k} A|^2)\left(-\theta|\nabla^{k+2} A|^2+ \widetilde C_{k+1} \left( \frac{|\nabla^{k+1} A|^2}{t^{1/2}} + \frac{|\nabla^{k+1} A|}{t^{\frac{k+2}{2}}}\right) \right) \notag \\ 
    &+ 8 t^{2k-1} |\nabla^{k} A| |\nabla^{k+1} A|^2 | \nabla^{k+2} A| \notag 
\end{align*}
Observe that by the induction hypothesis,
\begin{align*}
    8t^{2k-1}  |\nabla^{k} A| |\nabla^{k+1} A|^2 | \nabla^{k+2} A|&\le \widetilde C_{k} t^{2k-1} t^{-\frac{k-1}{2}} |\nabla^{k+1} A|^2 |\nabla ^{k+2} A|\\
    &= \widetilde C_{k} (t^{k-1/2}|\nabla^{k+1}A|^2)( t^{k-1/2}t^{-\frac{k-1}{2}}  |\nabla ^{k+2} A|)\\
    &\le \frac{\theta}{2} t^{2k-1}  |\nabla^{k+1} A|^4+ \widetilde C_{k} t^{k} |\nabla^{k+2} A|^2.   
\end{align*}
and 
$$ t^{k} \frac{|\nabla^{k+1} A|}{t^{\frac{k+2}{2}} } = \frac{1}{t} t^{\frac{k}{2}} |\nabla^{k+1}A| \le \frac{1}{2t} ( t^{k} |\nabla^{k+1} A|^2 +1). $$
Now choose $\mathcal A$ large so that $\mathcal A \theta > \widetilde C_{k}$. Hence 
\begin{align*}
    (\partial_t-\Delta^f)\psi\le & \frac{2k-1}{t}\psi - \frac{\theta}{2t} ( t^{k} |\nabla^{k+1} A|^2 )^2 + \frac{\widetilde C_{k}}{t}  \\
    &+\frac{\widetilde C_{k+1}}{t}  (\mathcal A+t^{k-1}|\nabla^{k} A|^2) \left( t^{k}|\nabla^{k+1} A|^2 + 1\right)\notag 
\end{align*}
and by the definition of $\psi$ and the induction hypothesis, 
\begin{equation*}
    \frac{1}{\mathcal A+L^2_k}\psi \le t^{k}|\nabla^{k+1}A|^2 \le \frac{1}{\mathcal A}\psi.
\end{equation*}
Thus we have the following estimate
\begin{equation*}
    (\partial_t-\Delta^f)\psi \le \frac{1}{t}\left(-\frac{1}{\widetilde C_{k+1} }\psi^2 + \widetilde C_{k+1}\right).
\end{equation*}
This inequality is all we need in order to apply the argument in the proof of \cite[Theorem 3.2]{ChenYin} to conclude
\begin{equation*}
    |\nabla^{k+1}A| \le \frac{L_{k+1}}{t^{\frac{k}{2}}} \ \ \text{ on } B_0 \left( x_0, \left( \frac 12 + \frac 1{2^{k+2}}\right) a\right) \times (0,T],
\end{equation*}
for some positive constant $L_{k+1}$. We have finished the proof by induction. 
\end{proof}

In the evolution equation (\ref{eqn evolution of A under extrinsic flow}) for the second fundamental form $A$, the term 
$$ \frac{\partial ^2 Q}{\partial h_{kl} \partial h_{pq}} \nabla_ih_{kl} \nabla_j h_{pq},$$
which is quadratic in $\nabla A$, prevents us from applying the standard estimates starting at $\nabla A$. This term is identically zero for the MCF, where $Q(h, g) = g^{ij} h_{ij}$. 

It is well-known that the concavity of $q$ can be used to better control the second fundamental form (see Remark 1 (1) in \cite{AMZ}). The following theorem is another manifestation of this observation.  

\begin{thm} \label{thm curvature estimates when F is convex and q concave}
    Let $F : M\times [0,T]\to \overline M$ be a convex solution to (\ref{eqn extrinsic flow}), where the curvature function is smooth, symmetric, monotone and concave. Assume that $F$ satisfies (\ref{eqn section 3 |A|le L}) and (\ref{eqn section 3 dist (kappa , partial Omega) ge delta}). For each $k = 1, 2, \cdots$, there is a positive constant $L^c_k$ depending on $k,n, L, \delta, q, a$ and $B^{loc}_j$ in (\ref{eqn |bar nabla^k bar R| le B_k}) for $j=0,1,\cdots, k+1$ such that 
    \begin{equation} \label{eqn curvature estimates when F is convex and q concave}
        |\nabla^k A|\le  \frac{L^c_k}{t^{k/2}}, \ \ \text{ on } B_0 \left( x_0, \frac{a}{2}\right) \times (0,T].  
    \end{equation}
\end{thm}

\begin{proof} The proof is very similar to that of Theorem \ref{thm higher order gradient estimates of A}, or \cite[Theorem 3.2]{ChenYin}, and we merely point out how the convexity of $F$ and convacity of $q$ are used: note that
\begin{align*}
    (\partial_t - \Delta^f) |A|^2 &= \partial_t g^{-1} * g^{-1} * A^2 + 2\langle (\partial_t - \Delta ^f )A, A\rangle - 2\Lambda^{pq} \langle \nabla_p A, \nabla_qA\rangle \\
    &\le \widetilde C_0  + 2\langle (\partial_t - \Delta ^f )A, A\rangle - 2\theta |\nabla A|^2. 
\end{align*}
Using (\ref{eqn evolution of A under extrinsic flow}), one has 
\begin{align*}
\langle (\partial_t - \Delta ^f )A, A\rangle = \frac{\partial^2 Q}{\partial h_{ij} \partial h_{kl}} (A, g) \nabla^p h_{ij} \nabla ^q h_{kl} A_{pq} + \mathscr F, 
\end{align*}
where 
$$ \mathscr F = \mathscr F(g^{-1}, A, Q(A, g), \partial_h Q (A, g), F^* \overline R, F^* \iota^2_{\mathfrak n} \overline R, F^* (\iota_{\mathfrak n} \bar\nabla \overline R)).$$
Hence 
\begin{align*}
\langle (\partial_t - \Delta ^f )A, A\rangle \le \frac{\partial^2 Q}{\partial h_{ij} \partial h_{kl}} (A, g) \nabla^p h_{ij} \nabla ^q h_{kl} A_{pq} + \widetilde C_0, 
\end{align*}
On the other hand, assume that $A$ is diagonalized at a point, then 
\begin{align*} 
\frac{\partial^2 Q}{\partial h_{ij} \partial h_{kl}} (A, g) \nabla^p h_{ij} \nabla ^q h_{kl} A_{pq} = \sum_p \kappa_p \frac{\partial^2 Q}{\partial h_{ij} \partial h_{kl}} (A, g) \nabla^p h_{ij} \nabla ^p h_{kl} \le 0
\end{align*}
since $\kappa_i \ge 0$ and $Q$ is concave by (\ref{eqn q concave iff Q concave}). Hence $\langle (\partial_t - \Delta ^f )A, A\rangle \le  \widetilde C_0$ and
\begin{align*}
    (\partial_t - \Delta^f )|A|^2 \le -2\theta |\nabla A|^2 + \widetilde C_0. 
\end{align*}
From here, together with Proposition \ref{prop evolution equation of nabla^k A}, one can obtain the higher derivative estimates (\ref{eqn curvature estimates when F is convex and q concave}) as in the proof of \cite[Theorem 3.2]{ChenYin}. 
\end{proof}

\section{Comparing two immersions}
In this section, let $F, \widetilde F : M\times [0,T]\to \overline M$ be two solutions to (\ref{eqn extrinsic flow}) which satisfy (\ref{eqn |A| le L}), (\ref{eqn dist (kappa, partial Omega) > delta}) and $F(\cdot, T)= \widetilde F (\cdot, T)$. In the following sections we derive estimates that are used in the proof of both Theorem \ref{thm backward uniqueness of extrinsic geometric flow} and Theorem \ref{thm backward uniqueness of convex extrinsic geometric flow with concave speed}. Hence we assume that $F$, $\widetilde F$ both satisfies the following condition: 
\begin{itemize}
    \item [(I)] either $F$, $\widetilde F$ satisfies (\ref{eqn |nabla A| le L}), or both $F$, $\widetilde F$ are convex ($A\ge0$ and $\widetilde A\ge 0$) and $q$ is concave. 
\end{itemize}

The following proposition follows directly from Theorem \ref{thm higher order gradient estimates of A} and Theorem \ref{thm curvature estimates when F is convex and q concave}. 

\begin{prop} \label{prop higher order curvature estimates on [e, T]}
Assume that $F$, $\widetilde F$ satisfies (\ref{eqn |A| le L}), (\ref{eqn dist (kappa, partial Omega) > delta}) and condition (I). Assume that the Riemann curvature tesnor $\overline R$ is uniformly bounded up to order $k+1$: 
\begin{equation} \label{eqn |nabla^j overline R| le B_j}
    |\overline \nabla^j \overline R|\le B_j, \ \ \text{ on } \overline M,
\end{equation}
for $j=0, 1, \cdots, k+1$. Then for all $e \in (0,T)$, there are positive constants $L_1^e, \cdots, L_k^e$ such that 
\begin{equation} \label{eqn higher order curvature estimates on [e, T]}
    |\nabla^j A | + |\widetilde \nabla^j \widetilde A|\le L_j^e,\ \ \text{ on } M\times [e, T]
\end{equation}
for $j=0, \cdots, k$. $L_j^e$ depends only on $k, n, L, \delta, q, e$ and $B_0, \cdots, B_{k+1}$. 
\end{prop}

We will be estimating the norms of various geometric quantities arising from $F, \widetilde F$, and we use the following convention: for quantities that only depends on $F$ (resp. $\widetilde F$) like $\nabla^k A$, (resp. $\widetilde \nabla^k A$), we use the norm calculated using $g$ (resp. $\tilde g$). For terms that involve both $F, \widetilde F$, for example $\widetilde A - A$ and $\widetilde \Gamma-\Gamma$, we use the norm calculated with respect to $g$. For example, 
$$ |\widetilde A - A|^2 =  g^{jl} g^{ik}(\tilde h_{ij} - h_{ij})(\tilde h_{kl} - h_{kl}).$$
For sections of the pullback bundles $N$, $\widetilde N$, we use the induced metric from $\bar g$. We also use $|\cdot |_0$ to denote the euclidean norms.

Next we recall the smooth homotopy between $F$ and $\widetilde F$ constructed in \cite{LeeMa} (see also \cite{SongWang}). Using (\ref{eqn extrinsic flow}) and that $F=\widetilde F$ at $t=T$, there is $\bar \epsilon>0$ depending on $q,\delta, L$ so that the following holds: 
\begin{itemize}
    \item [(i)] for all $(x, t) \in [T-\bar\epsilon, T]$, the distance between $F$ and $\widetilde F$ defined by 
\begin{equation*}
    d (x, t) = d_{\overline M} (F(x, t), \widetilde F(x, t))
\end{equation*}
satisfies 
\begin{equation} \label{eqn d < min given bar epsilon small}
    d < \min\left\{ i_0, 1, \frac{1}{\sqrt{2 B_0}}\right\},
\end{equation} 
\end{itemize}
here $i_0$ is the injectivity radius of $(\overline M, \bar g)$ and $B_0$ is the uniform upper bound of $\overline R$. We need (\ref{eqn d < min given bar epsilon small}) in order to apply some of the estimates in \cite{LeeMa}. 

Next, we recall the smooth homotopy between $F$ and $\widetilde F$ constructed in \cite{LeeMa} (see also \cite{SongWang}). By definition of $i_0$ and completeness of $(\overline M, \bar g)$, for all $(x, t)$ with $t\in [T-\bar\epsilon, T]$, there is an unique shortest geodesic $\gamma : [0,1] \to \overline M$ in $(\overline M, \bar g)$ such that 
$$ \gamma(0) = p:= F(x, t), \ \ \ \gamma(1) = \tilde p := \widetilde F(x, t). $$
For each $(x, t) \in M \times [T-\bar\epsilon, T]$, let $v\in T_pM$ be such that $\gamma(s) = \exp_p (tv)$. Note that $v$ is a section of the pullback bundle $N:=F^*T\overline M$, and the collections of these smooth geodesic 
\begin{equation} \label{eqn dfn of smooth homotopy gamma}
    \gamma : M\times [T-\bar\epsilon,T] \times [0,1] \to \overline M, \ \ (x, v, s)\mapsto \exp_{p} (sv)
\end{equation} 
forms a smooth homotopy between $F$ and $\widetilde F$.

In the following we work primarily on the interval $[T-\bar\epsilon, T]$. Let $e$ in Proposition \ref{prop higher order curvature estimates on [e, T]}) be chosen such that $e<T-\bar\epsilon$. We use $C$ to denote constants that depend only on $n, q,L, \delta, e$, and for each $k=0, 1\cdots$, $C_k$ are constants that depend on $n$, $q$, $L$, $\delta$ and $B_j$ for $j=0, \cdots, k+1$ (that is, the same dependence as $L_k^e$ in (\ref{eqn higher order curvature estimates on [e, T]})). Note that the constants $C, C_0, C_1, \cdots$ might change from line to line. 


Let
\begin{equation*}
    P = P(x, t) : T_{\tilde p} \overline M \to T_p \overline M
\end{equation*}
be the parallel transport along $-\gamma$. $P$ is a bundle map between the pullback bundles $\widetilde{N}:= \widetilde F^* T\overline M \to N$, and its action can be extended to 
$$T^{p,q}M \otimes \widetilde{N} \to T^{p,q}M \otimes N$$
by $P (X \otimes \tilde{\mathfrak s}):= X\otimes P\tilde{\mathfrak s}$, for all $(p, q)$-tensor $X$ on $M$ and $\tilde{\mathfrak s} \in \Gamma (M, \widetilde{N})$. The same setting was used in \cite{LeeMa}, where the authors study the uniqueness and backward uniqueness of mean curvature flow of any co-dimension in general ambient manifolds. 

First we recall the following estimates derived in \cite{LeeMa}. 

\begin{lem} \label{lem basic inequality from Lee Ma} 
Let $F, \widetilde F$ be as above. Then
\begin{itemize}
    \item [(1)] $|F_*|, |\widetilde F_*| \le C_0$, 
    \item [(2)] $|\tilde g^{-1}- g^{-1}|\le C |\tilde g-g|$, 
    \item [(3)] $|\tilde g-g| \le C |P\widetilde F_* - F_*|$,
    \item [(4)] $|D_tv| \le C_0 d ^2|F_t| + 2|P\widetilde F_t - F_t|$,  
    \item [(5)] $|D_t P|\le C_0 d (|F_t| + |P\widetilde F_t - F_t|)$.
\end{itemize}
\end{lem}

\begin{proof}
A proof of (1), (3) can be found in Lemma 2.1, Proposition 3.7 of \cite{LeeMa} respectively. (2) can be proved using (\ref{eqn uniform equivalence of metrics}). (4), (5) are proved in Proposition 3.1 and 4.2 of \cite{LeeMa} respectively. 
\end{proof}

From (2), (3) in Lemma \ref{lem basic inequality from Lee Ma}, we know that the term $P\widetilde F_* - F_*$ controls the difference $\tilde g -g$ and $ \tilde g^{-1} - g^{-1}$. The following lemma shows that it also controls the difference of the unit normal vectors. 

\begin{lem} \label{lem P tilde n-n le Ptilde F_* - F_*} One has 
\begin{equation} \label{eqn P tilde n-n le Ptilde F_* - F_*}
    |P\tilde{\mathfrak n} - \mathfrak n| \le C|P\widetilde F_* - F_*|.
\end{equation}
\end{lem}

\begin{proof}
It suffices to check the inequality at each point $x\in M$. Since $| P \widetilde F_* - F_*|$ is independent of coordinates chosen, we may assume that $g_{ij} = \delta_{ij}$ at $p$, $\bar g_{\alpha\beta} = \delta_{\alpha\beta} $ at $p = F(x, t)$. Hence 
\begin{equation*}
| P \widetilde F_* - F_*|^2 = | P\widetilde F_1 - F_1|^2 + \cdots + | P\widetilde F_n - F_n|^2. 
\end{equation*}

Since $P$ is a linear isometry, $\mathfrak n_1:= P\tilde{\mathfrak n}$ is normal to the the $n$-dimensional plane spanned by $P\widetilde F_1, \cdots, P\widetilde F_n$. For each $i=1, \cdots, n$, define
$$ e_i = F_i, \ \  v_i = P\widetilde F_i.$$
Also, define $e_{n+1} = \mathfrak n$. Identitying $T_p N$ with $\mathbb R^{n+1}$ with the positively oriented, orthonormal basis $(e_1, \cdots, e_{n+1})$, the inequality (\ref{eqn P tilde n-n le Ptilde F_* - F_*}) is equivalent to the following: Let $v_1, \cdots, v_n$ be linearly independent vectors in $\mathbb R^{n+1}$ such that $\mathfrak n_1$ is the unique unit normal of $\operatorname{span} (v_1,\cdots, v_n)$ with $\det (v_1, \cdots, v_n, \mathfrak n_1)>0$. Then there is $C_0>0$ such that  
\begin{equation} \label{eqn |n_1 - e_n+1|^2 le sum | v_i - e_i|^2}
| \mathfrak n_1 - e_{n+1}|^2 \le C \sum_{i=1}^n | v_i - e_i|^2. 
\end{equation}
To prove (\ref{eqn |n_1 - e_n+1|^2 le sum | v_i - e_i|^2}), it suffices to assume that $\epsilon^2:= \sum_i  | v_i - e_i|^2$ is small. Write 
\begin{equation*}
    v_i = (w_i, \alpha_i), \ \ \mathfrak n_1 = (\vec x, \beta), 
\end{equation*}
where $w_i, \vec x \in \mathbb R^n$ and $\alpha_i, \beta\in \mathbb R$. Since $\langle \mathfrak n_1 , v_i\rangle = 0$ for $i=1, \cdots, n$, we have 
$$w_i \cdot \vec x = -\beta \alpha_i \ \ i=1, \cdots n. $$
Let $A$ be the $n\times n$ matrix $ \begin{bmatrix} w_1^T & \cdots & w_n^T\end{bmatrix}$, then 
$$ A \vec x = -\beta \vec \alpha = -\beta \begin{bmatrix} \alpha_1 \\ \vdots \\ \alpha_n\end{bmatrix}. $$
Let $\mathscr E$ be the $n\times n$ matrix defined by $\mathscr E = I-A$, or $A = I-\mathscr E$. Then 
$$ |\mathscr E|^2 := \sum_{i,j=1}^n \mathscr E_{ij}^2 = \sum_i |v_i -e_i|^2 - |\alpha|^2 \le \epsilon.$$
By assuming that $\epsilon$ is small, $A$ is invertible and 
$$ A^{-1} = I + \mathscr E+ \mathscr E^2 + \cdots . $$
Using $\vec x = A^{-1}(-\beta \vec \alpha)$ and that $|\beta|\le 1$, we have 
$$ |x|\le |\beta| |\vec\alpha| (1+ \epsilon + \epsilon^2 + \cdots ) = \frac{\beta |\vec \alpha|}{1-\epsilon} \le 2\epsilon.$$
Since $1=|\mathfrak n_1|^2 = |\vec x|^2 + \beta^2$, we have 
$$1-|\beta|  \le \epsilon .$$ 
Lastly since $\det (v_1, \cdots, v_n, \mathfrak n_1)>0$ and $v_i$ are closed to $e_i$ for $i=1, \cdots, n$, we must have $\beta >0$ and hence $1-\beta <\epsilon$. As a result, 
$$ |\mathfrak n_1 - e_{n+1}|^2 = |\vec x|^2 + (1-\beta)^2 <5\epsilon^2$$
and this finishes the proof of the lemma. 
\end{proof}

Lastly, we recall the notion of the restriction of ambient tensors along $F$ and some of their properties. Given a $(p, q)$ tensor $S$ on $\overline M$ and $F: M\to \overline M$, let $S|_F$ be the restriction of $S$ onto the pullback bundle $N$. Hence $S|_F$ is a section on the bundle $N^{\otimes p} \otimes (N^*)^{\otimes q}$. Since $P : \widetilde N \to N$ is a bundle isomorphism, it induces a bundle isomorphism 
$$P^* : \widetilde N^{\otimes p} \otimes  (\widetilde N^*)^{\otimes q} \to N^{\otimes p} \otimes (N^*)^{\otimes q}.$$
In \cite[Lemma 4.2]{LeeMa}, the following estimates is proved: for any ambient tensors $S$ on $\overline M$, 
\begin{equation*}
    |P^* S|_{\widetilde F} - S|_F|\le \sup |\bar\nabla S| d. 
\end{equation*}
Indeed, from the proof in \cite{LeeMa}, the supremum on the right hand side can be taken over the image of the smooth homotopy (\ref{eqn dfn of smooth homotopy gamma}). Hence for each $k = 0,1,2,\cdots$ we have 
\begin{equation} \label{eqn estimtes difference of restrictions under parallel transport}
    |P^* (\bar\nabla ^k \overline R|_{\widetilde F}) - \bar\nabla ^k \overline R|_F | \le B_{k+1} d ,
\end{equation}
where $B_{k+1}$ is the bounds in (\ref{eqn |nabla^j overline R| le B_j}).

\section{Estimates of time derivatives}
Let $F, \widetilde F$ be solutions to the extrinsic geometric flow (\ref{eqn extrinsic flow}) which satisfy (\ref{eqn |A| le L}), (\ref{eqn dist (kappa, partial Omega) > delta}), condition (I) and $F(\cdot, T) = \widetilde F(\cdot, T)$. In this section, we estimate the time derivatives of several geometric quantities which measure the distance between $F$ and $\widetilde F$.

Since we are comparing two immersions, we need uniform bounds of $\partial^l_hQ$, $l=0, 1, \cdots$ not only on $K^\lambda_{SS}$, but also on a slightly enlarged neighborhood of $K^\lambda_{SS}$: Since $K^\lambda_{SS}$ is a compact subset in $\Omega_{SS}$ by Proposition \ref{prop K_{SS} is a compact set}, there is $\bar \delta>0$ depending only on $\delta, L, q$ such that the compact set 
\begin{equation} \label{eqn dfn of K_bar delta}
    K_{\bar \delta} = \{ (X, g)\in S(n) \times S_+(n) : \operatorname{dist} ( (X, g), K^\lambda_{SS} ) \le \bar\delta\}
\end{equation}
still lies inside $\Omega_{SS}$. Using Lemma \ref{lem evolution of A under extrinsic geometric flow} and the $k=2$ case in Proposition \ref{prop higher order curvature estimates on [e, T]}, there is $\bar C = C_2>0 $ such that $|\partial_t h_{ij}| \le \bar C$. Using that $F(\cdot, T) = \widetilde F(\cdot, T)$ and choosing a smaller $\bar\epsilon$ depending on $\bar C$ if necessary, one obtain the following 
\begin{lem}
    For each $x\in M$, there is a smooth oriented local chart $(U, (x^i))\in \mathscr A$ such that 
    \begin{equation} \label{eqn (h, g), (tilde h, tilde g) lies in the same balls}
      (h_{ij}(x, t), g_{ij}(x, t)), (\tilde h_{ij} (x, t), \tilde g_{ij} (x, t)) \in B_{\bar\delta} ((h_{ij} (x, T), g_{ij}(x, T))) \subset K_{\bar\delta}
    \end{equation}
    for all $t\in [T-\bar\epsilon, T]$, where $K_{\bar\delta}$ is defined in (\ref{eqn dfn of K_bar delta}).
\end{lem}
The above lemma implies the following estimates which will be frequently used in this and the next section.

\begin{lem} \label{lem eqn estimate of difference of partial^k Q(A,g)}
    For each $k =0, 1, \cdots$, there is a positive constant $C^Q_k$ depending on $k$, $n$, $L$, $\delta$, $q$ such that the following holds: for each $x\in M$, there is a local coordinates chart $(U, (x^i))$ around $x$ in $\mathscr A$ such that 
    \begin{equation} \label{eqn estimate of difference of partial^k Q(A,g)}
        \left| \partial^k_h Q(\tilde h_{ij}, \tilde g_{ij}) - \partial ^k_h Q (h_{ij} , g_{ij}) \right| \le C^Q_k ( |\widetilde A - A | + | P\widetilde F_* - F|). 
    \end{equation}
    for all $t\in [T-\bar\epsilon, T]$.
\end{lem}

\begin{proof}
    Since $K_{\bar\delta}$ is compact, for each $k$ there is $Q_k$ depending only on $k, L, \delta, q$ such that $|\partial^k Q| \le Q_k$ on $K_{\bar\delta}$. For each $x\in M$, let $(U, (x^i))$ be a smooth local coordinates around $x$ in $ \mathscr A$. Let $t\in [T-\bar\epsilon, T]$ be fixed. We write $h_{ij} = h_{ij} (x, t)$ and so on in this proof. By (\ref{eqn (h, g), (tilde h, tilde g) lies in the same balls}), the straight line in $\Omega_{SS}$ joining $(h_{ij}, g_{ij})$, $(\tilde h_{ij}, \tilde g_{ij})$ lies inside $K_{\bar\delta}$. Then one can apply the mean value theorem to $\partial ^k Q$ to conclude 
    \begin{equation*}
        |\partial^k_h Q(\tilde h_{ij}, \tilde g_{ij}) - \partial^k_h Q (h_{ij}, g_{ij}) |\le Q_{k+1} ( |\tilde h_{ij} - h_{ij}|_0 + | \tilde g_{ij} - g_{ij}|_0),
    \end{equation*}
    Together with (\ref{eqn g equivalent to delta_0}) and (3) in Lemma \ref{lem basic inequality from Lee Ma}, we obtain (\ref{eqn estimate of difference of partial^k Q(A,g)}). 
\end{proof}

Now we are ready to estimate the time derivatives. We start with the zeroth order term. 
\begin{lem} \label{lem estimates of D_tv} 
The vector fields $v$ along $F$ satisfies 
\begin{equation} \label{eqn estimates of D_t v}
       |D_t v| \le  C_1 (d^2 + | P\widetilde F_* -F_* | +  |\widetilde A- A|). 
\end{equation} 
\end{lem}

\begin{proof}
By (2) in Lemma \ref{lem basic inequality from Lee Ma},  
$$ |D_t v| \le  C_0 |F_t| d^2 +  |P \widetilde F_t - F_t|.$$
Since $\widetilde F_t = -\tilde f \tilde{\mathfrak  n}$, $F_t = -f \mathfrak n$, we have 
\begin{align*}
P \widetilde F_t - F_t &= -\tilde f P\tilde{\mathfrak  n} + f\mathfrak n \\
&=  -\tilde f (P\tilde{\mathfrak n} - \mathfrak n) - (\tilde f - f) \mathfrak n. 
\end{align*}
By (\ref{eqn estimate of difference of partial^k Q(A,g)}), Lemma \ref{lem basic inequality from Lee Ma} and the above inequality, we obtain (\ref{eqn estimates of D_t v}).
\end{proof}

Next, we estimate the first order quantities. By (2), (3) in Lemma \ref{lem basic inequality from Lee Ma} and Lemma \ref{lem P tilde n-n le Ptilde F_* - F_*}, the terms $\tilde g-g$, $\tilde g^{-1}-g^{-1}$ and $P\tilde{\mathfrak n} - \mathfrak n$ are bounded by $P\widetilde F_* - F_*$. For the latter quantity, we have the following estimates. 

\begin{prop} \label{prop estimates the time derivative of P tilde F_* - F_*} The time derivative of $P\widetilde F_* - F_*$ satisfies the estimate
\begin{equation} \label{eqn estimates the time derivative of P tilde F_* - F_*} 
    |D_t (P\widetilde F_* - F_*)| \le C_1( |v| + |P\widetilde F_* - F_*| + |\widetilde \nabla \widetilde A - \nabla A| + |\widetilde A - A|)
\end{equation}
\end{prop}

We need the following Lemma in proving the above Proposition.
\begin{lem} \label{lem estimates of times derivative of Ptilde nabla n - nabla n}
The quantity $P\widetilde \nabla \tilde{\mathfrak n} - \nabla \mathfrak n$ satisfies the estimate
    \begin{equation} \label{eqn estimates of times derivative of Ptilde nabla n - nabla n}
    |P\widetilde \nabla \tilde{\mathfrak n} - \nabla \mathfrak n|\le C_0 (|P\widetilde F_* - F_*|+|\widetilde A-A|). 
    \end{equation}
\end{lem}

\begin{proof} Using (4) in Lemma \ref{lem basic evolution of extrinsic flow}, for each $i=1, \cdots, n$, 
\begin{align*}
    P\widetilde \nabla_i \tilde{\mathfrak n} - \nabla_i \mathfrak n &= \tilde g^{jk} \tilde h_{ij} P \widetilde F_k - g^{jk} h_{ij} F_k \\
    &= \tilde g^{jk} \tilde h_{ij} (P\widetilde F_k - F_k) + (\tilde g^{jk} (\tilde h_{ij} - h_{ij}) + (\tilde g^{jk}-g^{jk})h_{ij} )F_k 
\end{align*}
Hence 
\begin{equation*}
    |P\tilde \nabla_i \tilde{\mathfrak n} - \nabla_i \mathfrak n|\le C_0 |\widetilde A| |P\widetilde F_* - F_*| + C_0 |\widetilde A-A|  + C_0 |A| |\tilde g-g|. 
\end{equation*}
Together with Lemma \ref{lem basic inequality from Lee Ma}, the Lemma is proved. 
\end{proof}

\begin{proof}[Proof of Proposition \ref{prop estimates the time derivative of P tilde F_* - F_*}]: for each $i=1, \cdots, n$,
    \begin{align*}
        D_t (P\widetilde F_i - F_i) &= (D_t P)\widetilde F_i + PD_t\widetilde F_i - D_tF_i \\ &=(D_t P)\widetilde F_i - (P\widetilde \nabla_i(\tilde f \tilde{\mathfrak n}) - \nabla_i (f\mathfrak n)) \\ &= (D_t P)\widetilde F_i - (P\widetilde \nabla_i \tilde f \tilde{\mathfrak n} - \nabla_i f \mathfrak n) - (P\tilde f \widetilde \nabla_i \tilde{\mathfrak n} - f\nabla_i \mathfrak n). 
    \end{align*}
    We bound each term on the right one by one. First, by (5) in Lemma \ref{lem basic inequality from Lee Ma} and $d<1$, 
    $$ |(D_t P) \widetilde F_i| \le C_0 (d + |P\widetilde F_* - F_*|). $$
    Next, note that
    \begin{align*}
       P\widetilde \nabla_i \tilde f \tilde{\mathfrak n} - \nabla_i f \mathfrak n &= \widetilde \nabla_i \tilde f P\tilde{\mathfrak n} - \nabla_i f \mathfrak n \\ &= P\tilde{\mathfrak n}(\widetilde \nabla_i \tilde f - \nabla_i f) + \nabla_if(P\tilde{\mathfrak n} - \mathfrak n)
    \end{align*}
    and 
    \begin{align*}
        \widetilde \nabla_i \tilde f - \nabla_i f &= \widetilde \nabla_i Q(\widetilde A,\tilde g) - \nabla_iQ(A,g) \\&= \frac{\partial Q}{\partial h_{jk}}(\widetilde A,\tilde g) \widetilde \nabla_i \tilde h_{jk} - \frac{\partial Q}{\partial h_{jk}}(A,g) \nabla_ih_{jk} \\&=\left(\frac{\partial Q}{\partial h_{jk}}(\widetilde A,\tilde g) - \frac{\partial Q}{\partial h_{jk}}(A,g)\right)\widetilde \nabla_i \tilde h_{jk} + (\widetilde \nabla_i \tilde h_{jk} - \nabla_i h_{jk})\frac{\partial Q}{\partial h_{jk}}(A,g).
    \end{align*}
    Using (\ref{eqn estimate of difference of partial^k Q(A,g)}) with $k=1$ we have the estimate
    \begin{equation*}
        \left|\frac{\partial Q}{\partial h_{jk}}(\widetilde A,\tilde g) - \frac{\partial Q}{\partial h_{jk}}(A,g)\right| \le C_2(|P\widetilde F_* - F_*| + |\widetilde A - A|).
    \end{equation*}
    Together with Lemma \ref{lem P tilde n-n le Ptilde F_* - F_*} we have 
    $$ |P\widetilde \nabla_i \tilde f \tilde{\mathfrak n} - \nabla_i f \mathfrak n| \le C_2(|P\widetilde F_* - F_*| + |\widetilde A - A| + |\widetilde \nabla \widetilde A - \nabla A|). $$
    Lastly, 
    \begin{align*}
       P\tilde f \widetilde \nabla_i \tilde{\mathfrak n} - f\nabla_i \mathfrak n &= \tilde f P \widetilde \nabla_i \tilde{\mathfrak n} - f\nabla_i \mathfrak n \\&= \tilde f(P\widetilde \nabla_i \tilde{\mathfrak n} - \nabla_i \mathfrak n) + (\tilde f - f)\nabla_i \mathfrak n.
    \end{align*}
    Using (\ref{eqn estimate of difference of partial^k Q(A,g)}) with $k=1$ Lemma \ref{lem estimates of times derivative of Ptilde nabla n - nabla n},
    $$ |P\tilde f \widetilde \nabla_i \tilde{\mathfrak n} - f\nabla_i \mathfrak n| \le C_1 ( |P\widetilde F_* - F_*| + |\widetilde A-A|). $$
    Combining the above inequalities, we obtain (\ref{eqn estimates the time derivative of P tilde F_* - F_*}).
\end{proof}

We use $\Gamma = (\Gamma_{ij}^k)$, $\widetilde \Gamma = (\Gamma_{ij}^k)$ to denote the Christoffel symbols of $g, \tilde g$ respectively. Note that $\widetilde\Gamma - \Gamma$ is a $(1,2)$ tensor on $M$. In the next two Propositions, we estimate the time derivative of $\widetilde\Gamma-\Gamma$ and $\nabla (\widetilde\Gamma-\Gamma)$. 

\begin{prop}\label{eqn estimates the time derivative of tilde Gamma - Gamma}
    The time derivative of $\widetilde \Gamma - \Gamma$ satisfies the estimate
    \begin{equation*}
        |\partial_t (\widetilde \Gamma - \Gamma)| \le C_1(|P\widetilde F_* - F_*| + |\widetilde \nabla \widetilde A - \nabla A| + |\widetilde A - A|).
    \end{equation*}
\end{prop}
\begin{proof}
    This is an immediate consequence of (2) in Lemma \ref{lem basic evolution of extrinsic flow}, Lemma \ref{lem basic inequality from Lee Ma} and (\ref{eqn estimate of difference of partial^k Q(A,g)}).
\end{proof}

Next we estimate the time derivative of the tensor $\nabla(\widetilde \Gamma - \Gamma)$. 
First by (\ref{eqn commutator of nabla and partial_t}),
\begin{equation} \label{eqn interchanging partial_t and nabla}
    \partial_t \nabla(\widetilde \Gamma - \Gamma) = \nabla \partial_t(\widetilde \Gamma - \Gamma) + (\partial_t \Gamma) * (\widetilde \Gamma - \Gamma).
\end{equation}
Denote 
\begin{align*}
    \widetilde K = \widetilde \nabla_i(\tilde f \tilde h_{jl}) + \widetilde \nabla_j(\tilde f \tilde h_{il}) - \widetilde \nabla_l(\tilde f \tilde h_{ij}),\\
    K = \nabla_i(fh_{jl}) + \nabla_j(fh_{il}) - \nabla_l(fh_{ij}).
\end{align*}
Then 
    \begin{align*}
    \partial_t(\widetilde \Gamma^k_{ij} - \Gamma^k_{ij}) &= \widetilde g^{kl}\widetilde K - g^{kl}K \\&= (\widetilde g^{kl} - g^{kl})\widetilde K - g^{kl}(\widetilde K - K). 
    \end{align*}
Together with (\ref{eqn interchanging partial_t and nabla}), 
    \begin{align*}
    \partial_t \nabla(\widetilde \Gamma - \Gamma) &= \nabla [(\widetilde g^{-1} - g^{-1}) * K_0 + (\tilde f - f) * K_1 + (\widetilde \nabla \tilde f - \nabla f) * K_2 \\
    &\ \ \ +(\widetilde A -A) * K_3 + (\widetilde \nabla \widetilde A - \nabla A) * K_4] + (\partial_t \Gamma) * (\widetilde \Gamma - \Gamma),
\end{align*}
where $K_i (i = 0, 1, 2, 3, 4)$ are smooth tensors. To deal with each terms in the above, We use
\begin{align*}
    \nabla (\tilde g^{-1} - g^{-1}) &= \nabla \tilde g^{-1} = (\widetilde \nabla - \nabla)*g^{-1} = (\widetilde \Gamma - \Gamma)*g^{-1},\\ \nabla(\tilde f - f) &= (\widetilde \nabla \tilde f - \nabla f) - (\widetilde \nabla - \nabla)\tilde f,\\ \nabla(\widetilde \nabla \tilde f -\nabla f) &= (\widetilde \nabla^2 \tilde f - \nabla^2 f) - (\widetilde \nabla - \nabla)\widetilde \nabla\tilde f
\end{align*}
and
\begin{align*}
    \widetilde \nabla_i\widetilde\nabla_j\tilde f - \nabla_i\nabla_jf &= \widetilde \nabla_i\widetilde\nabla_j\widetilde Q(\tilde h,\tilde g) - \nabla_i\nabla_jQ(h,g) \\
    &=\widetilde\nabla_i \left(\frac{\partial Q}{\partial h_{kl}}(\tilde h,\tilde g) \widetilde \nabla_j\tilde h_{kl}\right) - \nabla_i \left(\frac{\partial Q}{\partial h_{kl}}(h,g) \nabla_jh_{kl} \right) \\
    &=\left(\frac{\partial^2 Q}{\partial h_{kl}\partial h_{rs}}(\tilde h,\tilde g)\widetilde\nabla_j\tilde h_{kl}\widetilde\nabla_i\tilde h_{rs} - \frac{\partial^2 Q}{\partial h_{kl}\partial h_{rs}}(h,g)\nabla_j h_{kl}\nabla_i h_{rs}\right) \\
    &\ \ \ + \left(\frac{\partial Q}{\partial h_{kl}}(\tilde h,\tilde g) \widetilde \nabla_i \widetilde \nabla_j\tilde h_{kl} - \frac{\partial Q}{\partial h_{kl}}(h,g) \nabla_i \nabla_j h_{kl}\right).
\end{align*}
For the last term on the right hand side of the last equality, we have
\begin{align*}
    \frac{\partial Q}{\partial h_{kl}}(\tilde h,\tilde g) \widetilde \nabla_i \widetilde \nabla_j\tilde h_{kl} &- \frac{\partial Q}{\partial h_{kl}}(h,g) \nabla_i \nabla_j h_{kl} \\
    &= \left( \frac{\partial Q}{\partial h_{kl}}(\tilde h,\tilde g) - \frac{\partial Q}{\partial h_{kl}}(h,g) \right)\widetilde \nabla_i \widetilde \nabla_j\tilde h_{kl} \\
    &\ \ \ + \frac{\partial Q}{\partial h_{kl}}(h,g)((\widetilde \nabla_i - \nabla_i)\widetilde \nabla_j \tilde h_{kl} + \nabla_i(\widetilde \nabla_j\tilde h_{kl} - \nabla_jh_{kl})),
\end{align*}
the same trick using in the first term, together with (\ref{eqn estimate of difference of partial^k Q(A,g)}), Lemma \ref{lem basic evolution of extrinsic flow} and  Lemma \ref{lem basic inequality from Lee Ma}, one has 
\begin{prop} \label{eqn estimates the time derivative of nabla tilde Gamma - Gamma}
The time derivative of $\nabla(\widetilde \Gamma - \Gamma)$ satisfies the estimate
\begin{align*}
    |\partial_t \nabla(\widetilde \Gamma - \Gamma)| &\le C_3(|P\widetilde F_* - F_*| + |\widetilde \Gamma - \Gamma| + |\widetilde A - A| + |\nabla(\widetilde A - A)| \\ &+ |\widetilde \nabla \widetilde A - \nabla A| + |\nabla (\widetilde \nabla \widetilde A - \nabla A)|).
\end{align*}
\end{prop}

\section{The PDE-ODE inequality and the proof of the main Theorems}
In this section, we apply the general backward uniqueness theorem in \cite{KotschwarBU} to prove Theorem \ref{thm backward uniqueness of extrinsic geometric flow} and Theorem \ref{thm backward uniqueness of convex extrinsic geometric flow with concave speed}. As in \cite{LeeMa}, let $\mathcal X, \mathcal Y$ be the smooth vector bundles on $M$ given by 
\begin{equation} \label{eqn dfn of vb X and Y}
    \mathcal X = T^{0,2}M \oplus T^{0.3} M, \ \ \mathcal Y = N \oplus  ( T^{0,1}M \otimes N) \oplus T^{1,2}M \oplus T^{1,3}M. 
\end{equation}
Note that $\mathcal X$ is time-independent, while $\mathcal Y$ is time-dependent. On $\mathcal X$ one can take the time derivative $\partial_t ^{\mathcal X} := \partial _t \oplus \partial _t$, while on $Y$ one defines 
\begin{equation*}
    D_t ^{\mathcal Y} := D_t \oplus D_t \oplus \partial_t \oplus \partial_t. 
\end{equation*}

Recall that we have two solutions $F , \widetilde F : M \times [0, T]\to \overline M$ to the extrinsic flow (\ref{eqn extrinsic flow}) such that $F(\cdot, T) = \widetilde F(\cdot, T)$. From the discussion in section 3, there is $\bar \epsilon>0$ such that (\ref{eqn d < min given bar epsilon small}) is satisfied on $[T-\bar\epsilon, T]$, and hence the parallel transport from $\widetilde F$ to $F$ is well-defined. We remark that from (1) of Lemma \ref{lem basic evolution of extrinsic flow} and (\ref{eqn evolution of A under extrinsic flow}), $\bar\epsilon$ can be chosen independent of $T$. 

Let $X, Y$ be the sections of $\mathcal X, \mathcal Y$ respectively given by 
\begin{equation*}
    X = (\widetilde A - A) \oplus (\widetilde \nabla \widetilde A - \nabla A), \ \ \ Y = v \oplus (P\widetilde F_* - F_* )\oplus (\widetilde \Gamma - \Gamma) \oplus \nabla (\widetilde \Gamma - \Gamma). 
\end{equation*}
Note that $X$ contains terms that satisfy parabolic equations, while terms in $Y$ does not. 

By Lemma \ref{lem estimates of D_tv}, Proposition \ref{prop estimates the time derivative of P tilde F_* - F_*}, Proposition \ref{eqn estimates the time derivative of tilde Gamma - Gamma} and Proposition \ref{eqn estimates the time derivative of nabla tilde Gamma - Gamma}, we obtain the following Proposition. 

\begin{prop} \label{prop ODE ineq for Y}
    There exists a constant $C=C_2$ such that the section $X$ satisfies 
    \begin{equation} \label{eqn ODE ineq for Y}
    |D_t^{\mathcal Y} Y| \le C_2 (|X| + |\nabla X| + |Y|). 
    \end{equation}
\end{prop}

Next we estimate the terms $( \partial _t - \Delta^f) X$, or more precisely, 
$$ ( \partial_t- \Delta^f) (\widetilde \nabla^k \widetilde A - \nabla^k A)$$
for $k=0, 1$. 

\begin{prop} \label{prop estimate first order derivative of widetilde A and A}
    For any $k=0,1,2, \cdots$,
    \begin{align} \label{eqn estimate of widetilde nabla^k widetilde A - nabla^k A}
        |( \partial_t - \Delta^f)(\widetilde \nabla^k \widetilde A - \nabla^k A) | &\le C_{k+1}\bigg(d + |P\widetilde F_* - F_*| + |\widetilde \Gamma - \Gamma| + |\nabla(\widetilde \Gamma - \Gamma)| \\ &\ \ \ + \sum_{i=0}^k|\widetilde \nabla^i \widetilde A- \nabla^i A| + |\nabla(\widetilde \nabla^k \widetilde A - \nabla^k A)|\bigg).  \notag
    \end{align}
\end{prop}
\begin{proof}
    We start with the equation 
     \begin{align} \label{eqn estiamte of (partial t - Lambda)(widetilde nabla wildetilde A - nabla A)}
         (\partial_t - \Delta^f)(\widetilde \nabla^k \widetilde A - \nabla^k A) &= [ (\partial_t - \widetilde \Delta^{\tilde f} )\widetilde \nabla^k \widetilde A - (\partial_t - \Delta^f)\nabla^k A] + (\widetilde \Delta^{\tilde f} - \Delta^f)\widetilde \nabla^k \widetilde A \\& = (i) + (ii) \notag . 
     \end{align}
     Note that 
    \begin{align} \label{eqn difference of Delta's}
     \widetilde \Delta^{\tilde f} - \Delta^f &= \frac{\partial Q}{\partial h_{ij}} (\widetilde A, \tilde g) \widetilde\nabla_i\widetilde\nabla_j - \frac{\partial Q}{\partial h_{ij}} (A, g) \nabla_i\nabla_j \\
     &=\left(\frac{\partial Q}{\partial  h_{ij}}(\widetilde A, \tilde g) - \frac{\partial Q}{\partial h_{ij}} (A, g)\right)\widetilde\nabla_i\widetilde\nabla_j \notag \\ 
     &\ \ \ + \frac{\partial Q}{\partial h_{ij}} (A, g) ((\widetilde\nabla_i - \nabla_i)\widetilde \nabla_j + \nabla_i(\widetilde\nabla_j - \nabla_j)),\notag
    \end{align}
by (\ref{eqn estimate of difference of partial^k Q(A,g)}) with $k=1$ we have the estimates
     \begin{align} \label{eqn estimate of (ii)}
    |(ii)| \le C_{k+2} (|P\widetilde F_* - F_*| + |\widetilde A - A| + |\nabla(\widetilde \Gamma - \Gamma)| + |\widetilde \Gamma - \Gamma|). 
\end{align}
Next we estimate $(i)$. By Proposition \ref{prop evolution equation of nabla^k A}, 
$$ (\partial_t -\Delta^f) \nabla ^k A = \mathscr F_k, $$
where 
$$\mathscr F_k = \mathscr F_k ((\partial^l_hQ(A, g))_{l=0,1,\cdots,k+2}, g^{-1}, (\nabla^i A)_{i=0, 1, \cdots, k+1}, ( F^* (\iota_{\tilde{\mathfrak n}}^m\bar\nabla^j \overline R))_{j,m=0, 1, \cdots, k+1})$$
Hence 
\begin{align*}
    (i) &=  \mathscr F_k ((\partial^l_hQ(\widetilde A,\tilde g))_{l=0,1,\cdots,k+2}, \tilde g^{-1}, (\widetilde\nabla^i \widetilde A)_{i=0, 1, \cdots, k+1}, (\widetilde F^* (\iota_{\tilde{\mathfrak n}}^m\bar\nabla^j \overline R))_{j,m=0, 1, \cdots, k+1} ) \\
    &\ \ \ \ -\mathscr F_k ((\partial^l_hQ(A,g))_{l=0,1,\cdots,k+2}, g^{-1}, (\nabla^i A)_{i=0, 1, \cdots, k+1}, (F^*(\iota_{\mathfrak n}^m\bar\nabla^j \overline R))_{j,m=0, 1, \cdots, k+1} ). 
\end{align*}
Since $\mathscr F_k$ is multi-linear, one has 
\begin{align}\label{eqn first estimate of (i)}
    |(i)| &\le C_{k+1} \bigg( \sum_{l=0}^{k+2}|\partial^l_hQ(\widetilde A,\tilde g) - \partial^l_hQ(A,g)|+ |\tilde g^{-1} - g^{-1} | + \sum_{i=0}^{k+1} | \widetilde\nabla^i \widetilde A - \nabla^i A |  \\
    &\ \ \  + \sum_{j,m=0}^{k+1} \big|\widetilde F^* (\iota_{\tilde{\mathfrak n}}^m\bar\nabla^j \overline R) - F^*(\iota_{\mathfrak n}^m\bar\nabla^j \overline R)\big|\bigg). \notag
\end{align}

First, we estimate the last term on the right hand side. When $m=0$, the terms measures the difference between the pullback of an ambient tensor $\overline \nabla ^j \overline R$ along two maps $F, \widetilde F$. When $j=0$, one can rewrite the difference as 

\begin{align*}
    &(\widetilde F^*\bar R)_{ijkl} - (F^* \bar R)_{ijkl} \\&=  \bar R|_{\widetilde F}(\widetilde F_i,\widetilde F_j,\widetilde F_k,\widetilde F_l) - \bar R|_F(F_i,F_j,F_k,F_l) \notag \\ 
    &= P^* \bar R|_F (P \widetilde F_i,P \widetilde F_j, P \widetilde F_k, P \widetilde F_l) - \bar R(F_i, F_j, F_k, F_l)  \notag \\
    &=P^*(\bar R|_{\widetilde F})(P\widetilde F_i - F_i,P\widetilde F_j,P\widetilde F_k,P\widetilde F_l) + P^*\bar R|_{\widetilde F}(F_i,P\widetilde F_j - F_j,P\widetilde F_k,P\widetilde F_l) \notag \\ 
    &\ \ \ + P^*\bar R|_{\widetilde F}(F_i,F_j,P\widetilde F_k - F_k,P\widetilde F_l) + P^*\bar R|_{\widetilde F}(F_i,F_j,F_k,P\widetilde F_l - F_l) \notag \\ 
    &\ \ \ + (P^*\bar R|_{\widetilde F} - \bar R|_F)(F_i,F_j,F_k,F_l). \notag
\end{align*}
One can deal with the term $\widetilde F^*(\overline\nabla ^j\overline R) - F^* (\overline \nabla ^j \overline R)$ similarly, and by applying (1) in Lemma \ref{lem basic inequality from Lee Ma} and (\ref{eqn estimtes difference of restrictions under parallel transport}), we have 
\begin{equation} \label{eqn difference of Riemannian curvature tensors I}
|\widetilde F^*(\bar\nabla ^j  \overline R) - F^* (\overline \nabla ^j \overline R)|\le C( B_j |P\widetilde F_* - F_*| + B_{j+1} d). 
\end{equation}

When $m>0$, $F^* (\iota^m_{\mathfrak n} \bar\nabla ^j \bar R)$ is not a pullback of ambient tensors, however the estimates is similar:  e.g. for $m=1$ and $j=0$, 
\begin{align*} 
    &(\widetilde F^*(\iota_{\tilde{\mathfrak n}}\bar R))_{ijk} - (F^*(\iota_{\mathfrak n}\bar R))_{ijk} \\&=  \bar R|_{\widetilde F}(\tilde{\mathfrak n},\widetilde F_i,\widetilde F_j,\widetilde F_k) - \bar R|_F(\mathfrak n,F_i,F_j,F_k) \notag \\
    &= P^* \bar R|_{\widetilde F}(P\tilde{\mathfrak n},P\widetilde F_i,P\widetilde F_j,P\widetilde F_k) - \bar R|_F(\mathfrak n,F_i,F_j,F_k) \notag \\
    &=P^*\bar R|_{\widetilde F}(P\tilde{\mathfrak n} - \mathfrak n,P\widetilde F_i,P\widetilde F_j,P\widetilde F_k) + P^*\bar R|_{\widetilde F}(\mathfrak n,P\widetilde F_i - F_i,P\widetilde F_j,P\widetilde F_k) \notag \\ &\ \ \ + P^*\bar R|_{\widetilde F}(\mathfrak n,F_i,P\widetilde F_j - F_j,P\widetilde F_k) + P^*\bar R|_{\widetilde F}(\mathfrak n,F_i,F_j,P\widetilde F_k - F_k) \notag \\ &  \ \ \ + (P^*\bar R|_{\widetilde F} - \bar R|_F)(\mathfrak n,F_i,F_j,F_k). \notag
\end{align*}
Arguing similarly for $j>0$ and use (1) in Lemma \ref{lem basic inequality from Lee Ma} and Lemma \ref{lem P tilde n-n le Ptilde F_* - F_*}, one obtains
\begin{equation} \label{eqn difference of Riemannian curvature tensors II}
|\widetilde F^*(\iota^m_{\tilde{\mathfrak n}}\overline \nabla ^j  \overline R) - F^* (\iota^m_{\mathfrak n} \overline \nabla ^j \overline R)|\le C( B_j|P\widetilde F_* - F_*| + B_{j+1} d). 
\end{equation}

Next, the term $\widetilde \nabla^{k+1} \widetilde A - \nabla^{k+1} A$ can be written as 
\begin{equation*}
    \widetilde \nabla^{k+1} \widetilde A - \nabla^{k+1} A = (\widetilde \nabla - \nabla)\widetilde \nabla^k \widetilde A + \nabla(\widetilde \nabla^k \widetilde A - \nabla^kA).
\end{equation*}
Together with (\ref{eqn first estimate of (i)}), (\ref{eqn difference of Riemannian curvature tensors I}), (\ref{eqn difference of Riemannian curvature tensors II}), and (\ref{eqn estimate of difference of partial^k Q(A,g)}) one has
\begin{align} \label{eqn second estimate of (i)}
    |(i)| &\le C_{k+1}\bigg(|\tilde g - g| + |\tilde{g}^{-1} - g^{-1}| + |\widetilde \Gamma - \Gamma| \\ &+ \sum_{i=0}^{k}|\widetilde \nabla^i \widetilde A - \nabla^i A| + |\nabla(\widetilde \nabla^k \widetilde A - \nabla^k A)| + |P\widetilde F_* - F_*| + d \bigg). \notag
\end{align}
Now by Lemma \ref{lem basic inequality from Lee Ma}, (\ref{eqn estimate of (ii)}) and (\ref{eqn second estimate of (i)}) one obtains (\ref{eqn estimate of widetilde nabla^k widetilde A - nabla^k A}).
\end{proof}

Proposition \ref{prop estimate first order derivative of widetilde A and A} implies the following inequality.
\begin{prop} \label{prop PDE inequality for X}
    The section $X$ satisfies the PDE inequality
    \begin{equation} \label{eqn PDE inequality for X}
        \left | \left(\partial_t ^{\mathcal X} -\Delta^f \oplus \Delta^f \right) X \right | \le C_2(|X| + |\nabla X| + |Y|),
    \end{equation}
\end{prop} 

After all the preparations we are ready to prove Theorem \ref{thm backward uniqueness of extrinsic geometric flow}.

\begin{proof}[Proof of Theorem \ref{thm backward uniqueness of extrinsic geometric flow} and Theorem \ref{thm backward uniqueness of convex extrinsic geometric flow with concave speed}] 

We employ \cite[Theorem 3]{KotschwarBU} for the interval $[T-\bar\epsilon, T]$. Let $\tau = T - t$. Define $X_1(\tau) := X(\tau)$, $Y_1(\tau) := Y(\tau)$. First of all, by (\ref{eqn uniform equivalence of metrics}) and the completeness of $g(T)$, $g(\tau)$ is a complete metric on $M$ for all $\tau$. Note also that $\Lambda^{ij} = L_Q^{ij} (h, g)$ satisfies 
$$ \Lambda^{ij} \ge \theta g^{ij}. $$
for some positive $\theta$ by (\ref{eqn Lambda is equivalent to g}).

Next we check that (8) in \cite[Theorem 3]{KotschwarBU} is satisfied. First, denote $b=\partial_t g$. Then by (1) in Lemma \ref{lem basic evolution of extrinsic flow}, (\ref{eqn evolution of A under extrinsic flow}) and the bounds (\ref{eqn |bar nabla^k bar R| le B_k}), $|b|+ |\partial_\tau b|$ is uniformly bounded. 

The Riemann curvature tensors $R$ associated with the induced metrics $g_t$ are uniformly bounded by the Gauss equation, (\ref{eqn |A| le L}) and bounds on $\overline R$. Hence (9) in \cite[Theorem 3]{KotschwarBU} is satisfied, and $|\Lambda|$, $|\nabla \Lambda|$, $|\partial_t \Lambda|$ are uniformly bounded. As in \cite[Theorem 3]{KotschwarBU}, denote $\hat \nabla$ the connections induced on $\mathcal X$, then the bracket $[\partial_t, \hat \nabla] := \partial_t\hat\nabla - \hat\nabla\partial_t$ is a smooth section on the bundle $T^*M \oplus End(\mathcal X)$. Again by compactness $|[\partial_t, \hat \nabla]|$ is bounded. Then the inequality (8) in \cite[Theorem 3]{KotschwarBU} holds. By Proposition \ref{prop ODE ineq for Y} and Proposition \ref{prop PDE inequality for X}, $X_1$ and $Y_1$ satisfy the ODE-PDE inequalities (10) in \cite[Theorem 3]{KotschwarBU} on $[0,\bar\epsilon]$, here $\bar \epsilon$ is defined in section 3. By the backward uniqueness theorem \cite[Theorem 3]{KotschwarBU}, $X_1$ and $Y_1$ vanish on $[0,\bar\epsilon]$, that is $X$ and $Y$ vanish on $M^n \times [T-\bar \epsilon,T]$. Since $\bar \epsilon$ is independent of $T$, using the same argument we can show that $X, Y$ also vanish on $[T-2\bar\epsilon,T-\bar\epsilon]$. Argue inductively, we have shown that $X$ and $Y$ vanish identically on $M\times [e,T]$. In particular $d \equiv 0$ and hence $F = \widetilde F$ on $[e, T]$. Since $e <T-\bar\epsilon$ is arbitrary and $F, \widetilde F$ are continuous, we have $F = \widetilde F$ on $M\times [0,T]$ and this finishes the proof of Theorem \ref{thm backward uniqueness of extrinsic geometric flow} and Theorem \ref{thm backward uniqueness of convex extrinsic geometric flow with concave speed}. 
\end{proof}

\section{The compact case}
When $M$ is compact, it is probably of no surprise we can drop all assumptions on $(\overline M, \bar g)$ to prove the backward uniqueness:

\begin{thm} \label{thm backward uniqueness compact case}
    Let $(\overline M, \bar g)$ be an oriented $(n+1)$-dimensional Riemannian manifold, and let $M$ be a compact oriented $n$-dimensional manifold. Assume that $F , \widetilde F : M \times [0,T]\to \overline M$ are two solutions to the extrinsic geometric flow (\ref{eqn extrinsic flow}), where the curvature function is smooth, symmetric and monotone. If $F(\cdot, T) = \widetilde F(\cdot, T)$, then $F = \widetilde F$ on $M\times [0,T]$.  
\end{thm}

\begin{proof} [Sketch of proof] (\ref{eqn |A| le L}), (\ref{eqn |nabla A| le L}) and (\ref{eqn dist (kappa, partial Omega) > delta}) are satisfied since $M$ is compact. Indeed, by the compactness of $M\times [0,T]$, we have all higher derivative estimates of $A$, $\widetilde A$.

Again, by the compactness of $M$ there is $i_0 >0$ such that for all $(x, t) \in M \times [0,T]$, the exponential map of $(\overline M, \bar g)$ at $p = F(x, t)$ is a diffeomorphism from $\{ X\in T_pM : |X|_{\bar g} < i_0 \}$ onto its image. Since (\ref{eqn |A| le L}) and (\ref{eqn dist (kappa, partial Omega) > delta}) imply that $f$ is uniformly bounded, there is $\bar\epsilon >0$ so that (\ref{eqn d < min given bar epsilon small}) is satisfied on $[T-\bar\epsilon, T]$. Again using (\ref{eqn higher order gradient estimates of h}) we may choose $\bar\epsilon>0$ so that (\ref{eqn (h, g), (tilde h, tilde g) lies in the same balls}) is also satisfied. 

By our choice of $i_0$ and $\bar\epsilon$, the smooth homotopy $\gamma$ in (\ref{eqn dfn of smooth homotopy gamma}) is well-defined on $M\times [T-\bar\epsilon, T]$. Choosing a larger $B_k$ if necessary, we may assume that 
$$ |\overline\nabla^k \overline R| \le B_k$$
on the image of $\gamma$. These local bounds are sufficient to derive (\ref{eqn estimtes difference of restrictions under parallel transport}), which is then used to prove the time derivatives in section 5, and eventually the ODE-PDE inequalities (\ref{prop ODE ineq for Y}), (\ref{eqn PDE inequality for X}) in section 6. 
\end{proof}

\bibliographystyle{abbrv}

\vspace{0.75cm}

    \noindent Dasong Li\\
    Department of Mathematics, Southern University of Science and Technology \\
    No 1088, xueyuan Rd., Xili, Nanshan District, Shenzhen, Guangdong, China 518055 \\
    email: 12332876@mail.sustech.edu.cn

\bigskip

    \noindent John Man Shun Ma  \\
    Department of Mathematics, Southern University of Science and Technology \\
    No 1088, xueyuan Rd., Xili, Nanshan District, Shenzhen, Guangdong, China 518055 \\
    email: hunm@sustech.edu.cn

\end{document}